\documentclass[12pt,a4paper,final]{article}
\usepackage[applemac]{inputenc}
\usepackage{times,amsmath,amsthm,amssymb,tabularx,gastex,hhline,rotating,enumerate,xspace}
\usepackage[colorlinks,final]{hyperref}
\usepackage[usenames]{color}
\usepackage{showkeys}
\newcommand{\comment}[1]{}

\newenvironment{vd}{\noindent\color{blue} VD :  }{}

\newenvironment{am}{\noindent\color{red} AM : }{}

\newtheorem{theorem}{{\bf Theorem}}[section]
\newtheorem{corollary}[theorem]{{\bf Corollary}}
\newtheorem{definition}[theorem]{{\bf Definition}}
\newtheorem{example}[theorem]{{\bf Example}}
\newtheorem{lemma}[theorem]{{\bf Lemma}}
\newtheorem{proposition}[theorem]{{\bf Proposition}}
\newtheorem{remark}[theorem]{{\bf Remark}}

\newcommand{\refthm}[1]{Theorem~\ref{#1}}
\newcommand{\refcor}[1]{Corollary~\ref{#1}}

\newcommand{\reflem}[1]{Lemma~\ref{#1}}
\newcommand{\refprop}[1]{Proposition~\ref{#1}}
\newcommand{\refrem}[1]{Remark~\ref{#1}}

\newcommand{\refsec}[1]{Section~\ref{#1}}

\newcommand{\set}[2]{\left\{#1\mathrel{\left|\vphantom{#1}\vphantom{#2}\right.}#2\right\}}
\newcommand{\oneset}[1]{\left\{\mathinner{#1}\right\}}
\newcommand{\smallset}[1]{\left\{\mathinner{#1}\right\}}

\newcommand{\abs}[1]{\left|\mathinner{#1}\right|}

\newcommand{\bracket}[1]{\left[\mathinner{#1} \right]}

\newcommand{\gen}[1]{\left< \mathinner{#1} \right>}

\newcommand{\rdeg}[1]{\mathop{\mathrm{red}\mbox{-}\mathrm{deg}}(#1)}

\newcommand{\N}{\mathbb{N}}
\newcommand{\Z}{\mathbb{Z}}

\newcommand{\R}{\mathbb{R}}

\newcommand{\Zt}{\Z[t]}

\renewcommand{\phi}{\varphi}
\newcommand{\eps}{\varepsilon}

\newcommand{\alp}{\alpha}
\newcommand{\bet}{\beta}
\newcommand{\gam}{\gamma}
\newcommand{\del}{\delta}

\renewcommand\SS{\Sigma}

\newcommand\GG{\Gamma}

\newcommand\G{\mathcal{G}}

\newcommand{\cA}{\mathcal{A}}

\newcommand\MM{\mathcal{M}}
\newcommand{\BS}{\mathrm{\bf{BS}}}

\newcommand\ra{\longrightarrow}

\newcommand\RA[1]{\mathrel{\underset{#1}{\Longrightarrow}}}
\newcommand\LA[1]{{\underset{#1}{\Longleftarrow}}}
\newcommand\DA[1]{{\underset{#1}{\Longleftrightarrow}}}
\newcommand\OUT[4]{#1{\underset{#2}{\Longleftarrow}}#3{\underset{#2}{\Longrightarrow}}#4}

\newcommand\RAS[2]{\overset{#1}{\underset{#2}{\Longrightarrow}}}
\newcommand\LAS[2]{\overset{#1}{\underset{#2}{\Longleftarrow}}}
\newcommand\DAS[2]{\overset{#1}{\underset{#2}{\Longleftrightarrow}}}
\newcommand\OUTS[5]{#1\overset{#2}{\underset{#3}{\Longleftarrow}}#4\overset{#2}{\underset{#3}{\Longrightarrow}}#5}
\newcommand\INS[5]{#1\overset{#2}{\underset{#3}{\Longrightarrow}}#4\overset{#2}{\underset{#3}{\Longleftarrow}} #5}

\newcommand{\BG}{\text{\tiny Big}}

\newcommand{\smalloverline}[1]
{{\mspace{1mu}\overline{\mspace{-1mu}#1\mspace{-1mu}}\mspace{1mu}}}
\newcommand{\ov}[1]{\smalloverline{#1}}
\newcommand\ovv[2]{\overline{#1}\,\overline{#2}}

\newcommand\Ext{\mathrm{E}}

\newcommand\IRR{\mathop\mathrm{IRR}}

\newcommand\cdr{\mathop\mathrm{CDR}}

\newcommand{\Gred}{$G$-re\-du\-ced\xspace}

\newcommand{\IFF}{if and only if\xspace}
\newcommand{\homo}{homomorphism\xspace}

\newcommand{\nonarch}{non-Archi\-me\-dean\xspace}

\newcommand{\Breduced}{Britton-re\-du\-ced\xspace}

\newcommand{\WP}{Word Problem\xspace}

\newcommand{\gwp}{Generalized Word Problem\xspace}

\newcommand{\sse}{\subseteq}
\newcommand{\es}{\emptyset}
\newcommand{\sm}{\setminus}

\newcommand\ei[1]{{\emph{#1}\xspace}\index{#1}} 

\begin{document}

\title{Group extensions over infinite words\thanks{Part of the work has been started in 2007 when the authors where
at the CRM (Centro Recherche Matem{\`a}tica, Barcelona). 
It was finished when the first author stayed at Stevens Institute of Technology
in September 2010. The support of both institutions is greatly acknowledged}}

\author{Volker Diekert \and Alexei Myasnikov}


\date{\today}
\maketitle

\begin{abstract}
 We construct an extension $\Ext(A,G)$ of a given group $G$ by infinite \nonarch words over an discretely ordered abelian group like $\Z^n$. This yields an effective and uniform method to study various  groups that "behave like $G$". We show that the \WP for finitely generated subgroups in the extension is decidable \IFF  the Cyclic Membership Problem in $G$ is decidable.  The present paper embeds  the partial
monoid of infinite words as defined by Myasnikov, Remeslennikov, and Serbin  in  \cite{MRS05} into  $\Ext(A,G)$. Moreover,  we define the extension 
group $\Ext(A,G)$ for arbitrary groups $G$ and not only for free groups
as done in previous work. We show some structural results about 
the group (existence and type of torsion elements, generation by elements of order 2)
and we show that some interesting HNN extensions of $G$ embed naturally in the
larger group $\Ext(A,G)$.
\end{abstract}
\section{Introduction}\label{sec:intro}

In this paper we construct an extension  of a given group $G$ by infinite \nonarch words. 
The construction is effective and gives a new uniform method to study various  groups that "behave like $G$": limits of $G$ in the Gromov-Hausdorff topology, fully residually $G$ groups, groups obtained from $G$ by free constructions, etc. 
Infinite \nonarch words appeared first in \cite{MRS05} in connection with group actions on trees. The fundamentals for   group actions on simplicial trees (now known as Bass-Serre theory) were laid down by  Serre in his  
seminal book \cite{serre80}. 

 General $\Lambda$-trees for ordered abelian groups  
 $\Lambda$ were introduced by Morgan and Shalen in \cite{MoSh84} and their theory was further
developed by Alperin and Bass in  \cite{AB87}. 
 The Archimedean case concerns with group actions on $\R$-trees.

 A complete description of finitely generated groups acting freely on $\mathbb{R}$-trees was obtain in a series of papers \cite{BF,GaborioLP94}. It is known now as   Rips’ Theorem, see \cite{Chis01} for a detailed discussion. 
  
  For \nonarch actions much less is known. Much of the recent 
  progress is due to  Chiswell and M{\"u}ller \cite{ChMu},
   Kharlampovich, Myasnikov, Remeslennikov, and Serbin \cite{KMRS04,KMRS07,MRS05}
   and the recent thesis of Nikolaev \cite{nikolaev10thesis}. 
   In these papers groups acting on freely on $\Z^n$-trees are represented as 
   words where the length takes values in  the ring of integer polynomials 
   $\Zt$. More precisely, in \cite{MRS05} the authors represent 
   elements of Lyndon’s free $\Zt$-group $F^{\Zt}$ (the free group 
   with basis $\SS$ and exponentiation 
   in $\Zt$) 
   by infinite words, which are  defined as mappings  
   $w : [1,\alp] \to \SS^{\pm 1}$ over closed intervals $[1,\alp] = 
   \set{\bet \in \Zt}{1 \leq \bet \leq \alp
   }$. Here, the ring $\Zt$ is  viewed as an ordered abelian group
   in the standard way: $0<\alp$ if the leading coefficient of the 
   polynomial $\alp$ is positive. This yields  a regular free Lyndon length function with values in $\Zt$. 
   
   The importance of Lyndon’s group $F^{\Zt}$ became also prominent due to 
   its relation to algebraic geometry over groups and the the solution of 
   the Tarski Problems \cite{KMI98,KMII98,KMIII05,KMIV06}. 
   It was known by \cite{BMR95} and the results above
   that finitely generated fully residually free groups are embeddable into $F^{\Zt}$. The converse (every finitely generated subgroup of $F^{\Zt}$ is fully residually free) was shown in the original paper by Lyndon \cite{Lynpara60}.
      It follows that every finitely generated fully residually free group has a free length function with values in a free abelian group of finite rank with the lexicographic order.  It turned out that the representation of group elements as infinite words 
      over $\Zt$ is quite  intuitive and it enables a \emph{combinatorics on words} similar to finite words. This technique leads to the solution of various algorithmic problems for $F^{\Zt}$ using the standard Nielsen cancellation argument for the length function.
      
      This concept  is the starting point for our paper:  We  use finite words over $\SS^{\pm 1}$ to represent elements of $G$. Then, exactly as in the earlier papers mentioned above, an  infinite word is  
      a mapping $w : [1,\alp] \to \SS^{\pm 1}$ over a closed interval $[1,\alp] = 
   \set{\bet \in \Zt}{1 \leq \bet \leq \alp
   }$. The monoid of infinite words is endowed with a natural involution. We can read $w: [1,\alp] \to \SS^{\pm 1}$ {}from right-to-left and simultaneously we inverse each letter. This defines $\ov w$. Clearly,
   $\ov{\ov w} = w$ and $\ov{uv}= \ov v \, \ov u$. 
   The naive idea is to use  now $w \ov w = 1$ as defining relations in order to obtain a group. This idea falls short drastically, because  
   the group collapses. The image of the 
   $F(\SS)$ in this group is $\Z/2\Z$ (for $\SS \neq \es$). Therefore the set of 
   infinite words was viewed as a partial monoid, only. It was shown that 
    $F^{\Zt}$ 
   embeds into this partial monoid, but the proof is complicated and demands technical tools.   
   
   The first major deviation in our approach (from what has been done so far) is that we still work with equations
   $w \ov w = 1$, but we restrict them to freely reduced words $w$. 
   Just as in the finite case: A word $w$ is called freely reduced, if no factor $aa^{-1}$ (where $a$ is a letter) appears. This means, there is no $1 < \bet <\alp$ such that
   $w(\bet)= w(\bet+1)^{-1}$.   The submonoid generated by freely reduced words
   (inside the monoid of all infinite words) modulo defining equations 
    $w \ov w = 1$ defines a group (which is trivial) where $F(\SS)$ embeds (which
    is non-trivial). Actually,  $F^{\Zt}$ embeds. 
    It turns out that many freely reduced words satisfy $\ov w = w$. Thus, 
    the involution has fixed points, and 
    many elements have 2-torsion in our group. Actually, in natural situations the
    group is generated by these elements of order 2. 
    
Our focus is  more ambitious and goes beyond extending free groups $F(\SS)$. We begin with 
      an arbitrary group $G$ generated by  $\SS$. This gives rise to the notion 
      of a \Gred word. An $\Zt$-word is \Gred, if no finite factor
      $w[\bet, \bet+m]$ with $m \in \N$ represents the unit element 1 in $G$. 
      We let $R^*(A,G)$ denote the submonoid generated by \Gred words
   (inside the monoid of all infinite words)  where $A=\Zt$. Clearly, we may assume that $R^*(A,G)$
   contains all finite words (because we may assume that all letters are \Gred). 
   Then we factor out  defining equations for $G$ (which are words in $\SS^{\pm 1}$)
   and defining equations
    $u \ov u = 1$ with $u \in R^*(A,G)$. In this way we obtain a group denoted here by
    $\Ext(A,G)$.
    
    The first main result of the paper states that $G$ embeds into $\Ext(A,G)$, see \refcor{cor:embedG}. The result is obtained by the proof that
    some (non-terminating) rewriting  system is strongly confluent, thus confluent. 
    This is technically involved and covers all of \refsec{sec:csoverdoag}.
    
    The second main result concerns the question when the Word Problem is decidable 
    in all finitely generated subgroups of $\Ext(A,G)$. An obvious precondition is that the base group $G$ 
    itself must share this property. However, this is not enough and 
    makes the situation somehow non-trivial. 
    We show in \refcor{cor:main} that the Word Problem is decidable 
    in all finitely generated subgroups of $\Ext(A,G)$ \IFF 
    the Cyclic Membership Problem "$u \in \gen{v}$?" is decidable for all 
    $v \in G$. There are known examples where $G$ has a soluble Word Problem, but
    Cyclic Membership Problem is not decidable for some specific $v$, see \cite{olsap00,olsap01}. 
    On the other hand,  the Cyclic Membership Problem  is uniformly decidable in many natural classes (which encompasses classes of groups with decidable  
    Membership Problem w.r.t.~subgroups) like hyperbolic groups, 
    one-relator groups or effective HNN-extensions, see \refrem{cyclicmemb}.  
    
    In the final section we show that the partial monoid $\cdr(A,\SS)$  of infinite words with a cyclically reduced decompositions (c.f. \cite{MRS05}) embeds in our group $\Ext(A,G)$, 
    and we show that some interesting HNN extensions can be embedded into $\Ext(A,G)$
    as well which are not realizable inside the partial monoid $\cdr(A,\SS)$, \refprop{HNN}.  
    In order to achieve this result we show that every  cyclically \Gred word in $\Ext(A,G)$
    sits inside a free abelian subgroup of infinite rank, \refprop{abel}. 
    
    The proof techniques in this paper are of combinatorial flavor and 
    rely on the theory of rewriting systems.  
    No particular knowledge on \nonarch words or groups acting on $\Z^n$-trees 
    is required.

\section{Preliminaries on rewriting techniques}\label{sec:pre}
Rewriting techniques are a convenient tool to prove that certain constructions
have the expected properties. Typically we extend a given group by new generators
and defining equations and we want that the original group embeds in the 
resulting quotient structure.
For example,  HNN extensions and amalgamated products or Stalling's embedding 
(see \cite{Stallings71}) 
of a pregroup in its universal group can be viewed from this viewpoint, 
 \cite{ddm10}. Here we use them in the very same spirit. 
First, we recall the basic concepts.

 A \emph{rewriting relation} over a set $X$ is binary a relation
 $\RA{}\subseteq X \times X$. By $\RAS{+}{}$ ($\RAS*{}$ resp.) we mean the transitive (reflexive and transitive resp.) closure of  $\RA{}$.
By $\DA{}$ ($\DAS*{}$ resp.) we mean the symmetric (symmetric,
reflexive, and transitive resp.) closure of  $\RA{}$. We also write
$y \LA{}x$ whenever $x \RA{} y$, and we write $x \RAS{\leq k} {} y$
whenever we can reach $y$ in at most $k$ steps from $x$.
\begin{definition}
 The  relation
 $\RA{}\subseteq X \times X$ is called:
\begin{enumerate}[i.)]
\item \emph{strongly confluent}, if $\OUT y{}xz$ implies 
$\INS y{\leq 1}{}wz$ for some $w$,

\item \emph{confluent}, if $\OUTS y*{}xz$ implies 
$\INS y{*}{}wz$ for some $w$,

\item \emph{Church-Rosser}, if $y\DAS * {}z$ implies 
$\INS y{*}{}wz$ for some $w$,

\item \emph{locally confluent}, if $\OUT y{}xz$ implies 
$\INS y{*}{}wz$ for some $w$,

\item \emph{terminating}, if every infinite chain 
  \begin{align*}
    x_0 \RAS *{} x_1 \RAS *{} \cdots x_{i-1} \RAS *{} x_i  \RAS *{} \cdots 
  \end{align*}
becomes stationary,
\item \emph{convergent} (or \emph{complete}), if it is locally confluent
and terminating.
\end{enumerate}
\end{definition}

The following facts are well-known, proofs are easy and 
can be found in any text
book
on rewriting systems, see e.g. \cite{bo93springer,jan88eatcs}.

\begin{proposition}\label{frida}
The following assertions hold: 
\begin{enumerate}[1.]
\item Strong confluence  implies confluence.
\item Confluence  is equivalent with Church-Rosser.
\item Confluence implies local confluence, but the converse is false,
  in general.
\item Convergence (i.e., local confluence and termination together) 
implies confluence.
\end{enumerate}
\end{proposition}

Often one is  interested in the case, only where  $X$ is a free group or a  free monoid and the 
rewriting relation is specified by directing defining equations. Here we are
more general in the following sense. 
Let $M$  be any monoid. 
A \emph{rewriting system
} over $M$ is a relation
 $S \subseteq M \times M$. Elements $(\ell,r)\in S$ are also called \ei{rules}. The system $S$ defines the rewriting relation
 $\RA S \subseteq M \times M$ by 
\begin{align*}
 x   \RA S y, \;\mbox{if } \; x=p\ell q, \; y= prq \; \mbox{ for some rule } \; (\ell,r)
 \in S.
 \end{align*}
The relation  $\DAS * S \subseteq M \times M$ is a congruence, hence
the congruence classes form a monoid which is denoted by 
 $M/\set{\ell=r}{(\ell,r)\in S} $. Frequently
we simply write $M/S $ for this quotient monoid. Notice, that if $M$ is a free monoid with basis $X$ then $M/S $ is the monoid given by the presentation $\langle X \mid \ell=r, \mbox{ where }  (\ell,r)\in S\rangle$.

We say that $S$ is strongly confluent or confluent etc, if in fact
$\RA S$ has the corresponding property. Instead of $(\ell,r) \in S$ we
also write $\ell {\underset{}{\longrightarrow}} r \in S$ and $\ell
{\underset{}{\longleftrightarrow}} r \in S$ in order to indicate that
both $(\ell,r) \in S$ and $(r,\ell) \in S$.
By $\IRR(S)$ we mean the set of \ei{irreducible normal forms}.
This is the subset of $M$ where no rule of $S$ can be applied, i.e.,
$$\IRR(S) = M \sm \bigcup_{(\ell, r ) \in S} M\ell M.$$

If $S$ is terminating, then we have $1 \in \IRR(S)$, and if 
$S$ is convergent, then the canonical \homo $M \to M/S$ induces
a bijection between $\IRR(S)$ and the quotient monoid  $M/S$.

If a quotient monoid is given by a finite convergent string rewriting
system $S \subseteq \GG^* \times \GG^*$, then the monoid has a
decidable \WP, which yields a major
interest in  these systems.

\begin{example}\label{freegroup}
Let $\SS$ be a set and $\SS^{-1}$ be disjoint copy. 
Then the set of rules $\set{aa ^{-1}\ra 1, a^{-1}a\ra 1}{a \in \SS}$
defines a strongly confluent and terminating system over $(\SS \cup\SS^{-1})^*$
which defines the \ei{free group} $F(\SS)$ with basis $\SS$.
\end{example}

In this paper however,
we will deal mainly with non-terminating systems which are moreover 
in many cases infinite. So convergence plays a minor role here. 
There is another class of string rewriting systems which for finite systems leads
to a polynomial space (and hence exponential time in the worst case)  decision algorithm for the \WP.

\begin{definition}\label{def:preperfect}
 A  string rewriting
system $S \subseteq \GG^* \times \GG^*$  is called
\emph{pre-perfect}, if the following three conditions hold:
\begin{enumerate}[1.]
\item The system $S$ is confluent.
\item If we have $\ell \longrightarrow r \in S$, then we have 
$\abs{\ell} \geq \abs{r}$ where $\abs{x}$ denotes the length of a word
$x$.
\item If we have $\ell \longrightarrow r \in S$ with $\abs{\ell} =
  \abs{r}$,
then we have $r \longrightarrow \ell \in S$, too.
\end{enumerate}
\end{definition}

Clearly, a convergent length-reducing system is pre-perfect,
and if a confluent system satisfies $\abs{\ell} \geq \abs{r}$ 
for all $\ell \longrightarrow r \in S$, then we can add symmetric 
rules in order to make it  pre-perfect.

\section{Non-{A}rchimedean words}
We consider group extensions over infinite words of a specific type. 
These words are also called \ei{\nonarch words}, because 
they are defined over non-ar\-chi\-me\-de\-an ordered abelian groups.

\subsection{Discretely ordered abelian groups}

A \ei{ordered abelian group} is an abelian  group $A$ 
together with a linear order $\leq$ such that $x\leq y$ \IFF 
$x+z\leq y+z$ for all $x,y,z \in A$. It is \ei{discretely ordered}, if an addition there is least 
positive element $1_A$. Here, as usual, an element $x$ is \emph{positive}, 
if $0< x$. An ordered abelian group is \ei{Archimedean}, if for all 
$0 \leq a \leq b$ there is some $n \in \N$ such that $b < na$, otherwise it is \ei{\nonarch.}    

If $B$
is any ordered abelian group, then $A=\Z \times B$ is discretely ordered
with $1_A = (1,0)$ and the lexicographical ordering:
\begin{equation*}
(a,b) \leq (c,d) \mbox{ if } b < d \mbox{ or } b=d \mbox{ and } a < c.
\end{equation*}
The group is \nonarch unless $B$ is trivial since
$(n,0) < (0,x)$ for all $n \in \N$ and positive $x \in B$.

In particular, $\Z\times \Z$ is a \nonarch discretely ordered abelian  group. 
It serves as our main example. Iterating the process 
all finitely generated free abelian $\Z^k$ are viewed as being discretely ordered; and by
a transfinite iteration we can consider 
arbitrary direct sums of $\Z$. 
This is where we limit ourselves. 
In this paper we consider discretely ordered abelian groups only, which can be written as
\begin{equation}
 A = \oplus_{i \in \Omega} \gen{t_i},
\end{equation}
where $\Omega$ is a set of ordinals, and $\gen{t_i}$ denotes the infinite 
cyclic group $\Z$ generated by the element $t_i$.
Elements of $A$ are finite sums $\alpha = \sum_i n_i t_i$ with $n_i \in \Z$.
Since the sum is finite, either $\alpha = 0$ or there is a greatest ordinal $i \in \Omega$ (denoted by $\deg(\alpha)$) 
with $n_i \neq 0$. By convention, $\deg(0) = -\infty$. 
We call $\deg(\alpha)$ the \emph{degree} or \emph{height} of $\alpha$. An element $\alpha = \sum_i n_i t_i\in A$ is 
called \emph{positive}, if $n_d > 0$ for $d=\deg(\alpha)$. We let
$\alpha \leq \beta$, if $\alpha = \beta$ or $\beta - \alpha$ is positive. 
Moreover, for $\alpha, \beta \in A$
we define the \emph{closed interval} $[\alpha,\beta]=\set{\gamma \in A}{\alpha \leq \gamma \leq \beta}$. Its \ei{length} is defined to be $\bet - \alp+1$. 

For $\Z \times 
\Z$ the interval $[(-3,0),(2,1)]$ is  depicted as in Fig.~\ref{figci}. Its length is $(6,1)$. 

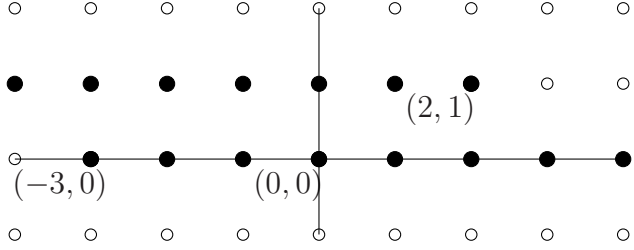
\begin{figure}[p,t,b]
  \centering
\begin{picture}(80,50)(10,-100)
\gasset{Nh=1.5,Nw=1.5,Nfill=n}
{\gasset{ExtNL=y,NLangle=-140,NLdist=0.3}
\multiput(0,-40)(0,-10){4}{
\multiput(0,0)(10,0){9}{\node(n0)(0,-20){}}}
\gasset{Nh=2,Nw=2,Nfill=y}

\node[Nfill=n](n0)(40,-80){$(0,0)$}
\node[Nfill=y](n01)(10,-80){$(-3,0)$}

\drawline[AHnb=0](40,-60.0)(40,-90){}
\drawline[AHnb=0](0,-80)(80,-80){}

\multiput(10,-80)(10,0){8}{\node[Nfill=y](n0)(0,0)}
\multiput(0,-70)(10,0){7}{\node[Nfill=y](n0)(0,0)}

\node[Nfill=y](ne)(60,-70){$(2,1)$}

}
\end{picture}
\caption{A closed interval of length $(6,1)$ in $\Z \times \Z$}\label{figci}
\end{figure}

Sometimes we simply  illustrate intervals of length $(m,1)$ as $[\cdots)(\cdots]$
and intervals of length $(m,2)$ as $[\cdots)(\;\cdots\; )(\cdots]$. This will become 
clearer later.

\subsection{Non-{A}rchimedean words over a group $G$}
An \ei{involution} of a set $M$ is a mapping $M \to M, x \mapsto \ov x$ with 
$\ov {\ov x} = x$ for all $x\in M$. A \ei{monoid with involution} is a monoid $M$ with an involution $x \mapsto \ov x$  such that $\ov {xy} = \ov y \ov x$ 
for all $x,y \in M$ and, as a consequence, $\ov 1 = 1$. Every group is a monoid with involution 
$x \mapsto  x^{-1}$.    Obviously, if $M$ is a monoid with involution $x \mapsto \ov x$ then the quotient $M/\set{x \ov x = 1}{x \in M}$ is a group. Furthermore, if $G$ is a group and $M$ is  a monoid with involution then every monoid homomorphism respecting involutions $\phi: M \to G$  factors through this canonical quotient.  Let $a \mapsto \ov a$ denote a bijection between 
sets 
$\SS$ and $\ov{\SS}$, hence  $\ov{\SS} =  \{\ov{a} \mid a \in \SS\}$. The map $a \mapsto \ov{a}, \ov{a} \mapsto a$ is an involution on $\SS \cup \ov{\SS}$ with  $\ov{\ov{ a}} = a$.  It extends to an  involution $x \mapsto \ov{x}$ on  the free monoid $(\SS \cup \ov{\SS})^*$ with basis $\SS \cup \ov{\SS}$ by $\overline{a_1\cdots a_n}=
\overline{a_n} \cdots \overline{a_1}$. 
In case that  $\SS \cap \ov{\SS} = \es$ 
the resulting structure $((\SS \cup \ov{\SS})^*, \cdot, 1 , \ov{\phantom{a}})$ is the {\em free monoid with involution} with basis $\SS$.

Throughout $G$ denotes a group with  a generating set 
$\Sigma$. We always assume that $a \neq 1$ for all $a \in \Sigma$.  We let $ \GG = \Sigma \cup \overline{\Sigma}$, where
$\overline{\Sigma} = \Sigma^{-1} \sse G$ and  $\overline{a} = a^{-1}$ for $a \in \GG$. 
The inclusion $\GG \sse G$ induces  the
 canonical homomorphism (presentation) onto the group $G$:
\begin{equation*}
 \pi: \GG^* \to G.
\end{equation*}

Clearly, for every word $w \in \GG^*$
we have $\pi(\overline{w}) = \pi(w)^{-1}$. Note that there are fixed points for the involution on $\GG$ in case $\SS$ contains an element of order 2.

Let $A = \oplus_{i \in \Omega} \gen{t_i}$ be a  discretely ordered abelian group as above.
A \emph{partial $A$-map} is a map $p: D \rightarrow \GG$ with
$D \subseteq A$. Two partial maps $p: D \rightarrow \GG$ and $p': D' \rightarrow \GG$  are termed equivalent if $p'$ is an {\em $\alpha$-shift} of $p$ for some $\alpha \in A$, i.e., $D'=\set{\alpha+\beta}{\beta \in D}$ and $p'(\alpha+\beta) = p(\beta)$ 
for all $\beta \in D$.  
This an equivalence relation on partial $A$-maps, and an equivalence class 
of partial $A$-maps is called a \emph{partial $A$-word}. If $D = [\alpha,\beta]=\set{\gamma \in A}{\alpha \leq \gamma \leq \beta}$ then the equivalence class of 
$p:[\alpha,\beta]\rightarrow \GG$ is called a \emph{closed $A$-word}.  By abuse of language a closed (resp. partial) $A$-word is sometimes simply called a \ei{word} (resp. \ei{partial word}). 

A word $p:[\alpha,\beta]\rightarrow \GG$ is  \ei{finite} if the set $[\alpha,\beta]$ is finite, otherwise
it is \ei{infinite}. Usually, we identify finite words  with the corresponding 
elements in $\GG^*$. 

If $p:[\alpha,\beta]\rightarrow \GG$ and 
$q:[\gamma,\delta]\rightarrow \GG$ are closed $A$-words, then we define their concatenation as follows. We may 
assume that $\gamma = \beta+1$ and we let:
\[
\begin{array}{rcll}
p \cdot q: [\alpha,\delta]& \to & \GG\\
x& \mapsto& p(x)& \mbox{ if } x \leq \beta\\
x &\mapsto& q(x)& \mbox{ otherwise}.
\end{array}
\]
It is clear that this operation is associative. Hence, the set of closed $A$-words forms a monoid, which we denote by $W(A,\GG)$. The 
neutral element, denoted by $1$, is the totally undefined mapping.
 The standard 
representation of an $A$-word $p$ is a mapping $p: [1, \alpha] \to \GG$, where $0 \leq \alpha$. In this case 
$\alpha$ is called the \emph{length} of $p$; sometimes we also write $|p|=\alpha$. The \emph{height} or \emph{degree} 
of $p$ is the degree of $\alpha$; we also write $\deg(p) = \deg(\alpha)$. 
For a partial word $p: D \rightarrow \GG$ and $ [\alpha,\beta]\sse D$ we denote by $p[\alpha,\beta]$ the restriction of $p$ to the interval $[\alpha,\beta]$.
Hence $p[\alpha,\beta]$ is a closed word. Sometimes we write $p[\alpha]$
instead of $p[\alpha,\alp]$. Thus,  $p[\alpha]= p(\alp)$.

The monoid $W(A,\GG)$  is a monoid with involution  $p \mapsto \ov{p}$ where for $p:[1,\alpha] \rightarrow 
\GG$ we define $\overline{p} \in W(A,\GG)$ by $\overline{p}: [-\alpha,-1] \rightarrow \GG$, $-\beta \mapsto \ov{p(\beta)}$.\\

Recall that $A = \oplus_{i \in \Omega} \gen{t_i}$. We may assume that $0$ is the least ordinal in $\Omega$, in which case $\Z$ can be  viewed as a subgroup of $A$ via the 
embedding $n \mapsto nt_0$. Thus $1\in \N$ is also the smallest positive element in $A$. If, for example,  $A = \Z \times \Z$, 
then we have identified $1\in \N$ with the pair $(1,0)$. 

If $x \in W(A,\GG)$ and $x = pfq$ for some $p, q \in W(A,\GG)$ then $p$  is called a \emph{prefix}, $q$  is called a \emph{suf\-fix}, and  $f$ is called a \emph{factor} 
of $x$. If $1 \neq f \neq x$ then $f$ is  called a \emph{proper factor}.
As usual, a factor  is finite, if $|f| \in \N$. Thus, a finite factor can be written as $x[\alpha,\beta]$ 
where $\beta=\alpha+n$, $n \in \N$. 

A closed word $x: [1,\alpha] \rightarrow \GG$ is called
\emph{freely reduced} if $x(\beta) \neq \overline{x(\beta+1)}$ for all
$1 \leq \beta < \alpha$. It is called \emph{cyclically reduced} if 
$x^2$ is freely reduced. 

As a matter of fact we need a stronger conditions. The word $x$ is called \emph{\Gred}, if no finite
factor $x[\alpha,\alpha+n]$ with $n\in \N$, $n \geq 1$, becomes the
identity $1$ in the group $G$. Note that all \Gred words are freely reduced
by definition. We say $x$ is \emph{cyclically \Gred}, 
if every finite power $x^k$ with $k \in \N$ is \Gred. 
 Over a free group $G$ with basis $\Sigma$  a word is freely reduced if and only if it is \Gred, and it is cyclically \Gred \IFF it is cyclically reduced.

 In Fig.~\ref{figone} we see a closed word which is not freely reduced.
 Fig.~\ref{figw} defines a word $w$ with a  sloppy notation 
 $[aaa\cdots)(\cdots abab\cdots)(\cdots bbb]$. 
 Fig.~\ref{figwa} shows that for the same word $w$ we have $aw \neq wb$
 (because $aw[(0,1)] = a$ and $wb[(0,1)] = b$),
 but we have $aaw = wbb$ in the monoid $W(A,\GG)$, see Fig.~\ref{figwaa}. Recall, that two elements $x, y$ in a monoid $M$ are called \emph{conjugated}, if $xw = wy$ for some $w \in M$. 
  Fig.~\ref{figxw} shows that all finite words $x,y \in \GG^*$ are conjugated in
  $W(A,\GG)$ provided they have the same length $\abs{x} = \abs {y}$ and $A$ is \nonarch. Indeed 
  $t = [uuu\cdots)(\cdots
  vvv]$ does the job  $ut =  tv$. Clearly, $ut =  tv$ implies $\abs{x} = \abs {y}$. In particular, this shows that the monoid $W(A,\GG)$ is not free.  Indeed, if $x$ and $y$ are conjugated elements in  a free monoid, say 
 $xw = wy$,   
 then $x = rs $, $y= sr$, and $w = (rs)^mr $ for some $r,s \in \GG^*$ and $m\in \N$, which is not the case for the example above.


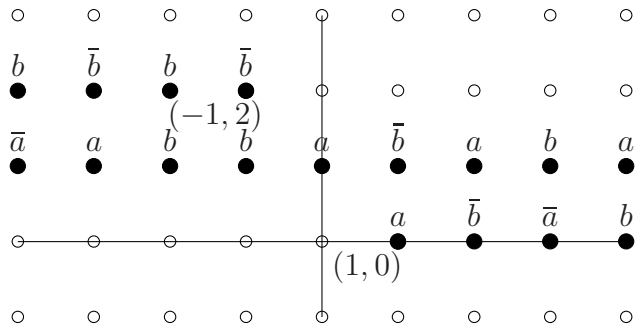
\begin{figure}[t,b,p]
  \centering
\begin{picture}(80,60)(10,-100)
\gasset{Nh=1.5,Nw=1.5,Nfill=n}
{\gasset{ExtNL=y,NLangle=-140,NLdist=0.3}
\multiput(0,-30)(0,-10){5}{
\multiput(0,0)(10,0){9}{\node(n0)(0,-20){}}}
\gasset{Nh=2,Nw=2,Nfill=n}

\node[Nfill=y](n0)(50,-80){$(1,0)$}
\drawline[AHnb=0](40,-50.0)(40,-90){}
\drawline[AHnb=0](0,-80)(80,-80){}

\multiput(50,-80)(10,0){4}{\node[Nfill=y](n0)(0,0)}
\multiput(0,-70)(10,0){9}{\node[Nfill=y](n0)(0,0)}
\multiput(0,-60)(10,0){4}{\node[Nfill=y](n0)(0,0)}
\node[Nfill=y](ne)(30,-60){$(-1,2)$}

\gasset{Nfill=n,Nframe=n,ExtNL=y,NLangle=90,NLdist=1}
\node(n1)(50,-80){$a$}
\node(n1)(60,-80){$\ov b$}
\node(n1)(70,-80){$\ov a$}
\node(n1)(80,-80){$b$}
\node(n1)(00,-70){$\ov a$}
\node(n1)(10,-70){$a$}
\node(n1)(20,-70){$b$}
\node(n1)(30,-70){$b$}
\node(n1)(40,-70){$a$}
\node(n1)(50,-70){$\ov b$}
\node(n1)(60,-70){$a$}
\node(n1)(70,-70){$b$}
\node(n1)(80,-70){$a$}
\node(n1)(00,-60){$b$}
\node(n1)(10,-60){$\ov b$}
\node(n1)(20,-60){$b$}
\node(n1)(30,-60){$\ov b$}
}
\end{picture}
\caption{A closed non-freely reduced word of length $(-1,2)$.}
  \label{figone}
\end{figure}


\begin{figure}[t,b,p]
  \centering
\begin{picture}(80,60)(10,-100)
\gasset{Nh=1.5,Nw=1.5,Nfill=n}
{\gasset{ExtNL=y,NLangle=-140,NLdist=0.3}
\multiput(0,-30)(0,-10){5}{
\multiput(0,0)(10,0){9}{\node(n0)(0,-20){}}}
\gasset{Nh=2,Nw=2,Nfill=y}

\node[Nfill=y](n0)(50,-80){$(1,0)$}
\drawline[AHnb=0](40,-50.0)(40,-90){}
\drawline[AHnb=0](0,-80)(80,-80){}

\multiput(50,-80)(10,0){4}{\node[Nfill=y](n0)(0,0)}
\multiput(0,-70)(10,0){9}{\node[Nfill=y](n0)(0,0)}
\multiput(0,-60)(10,0){5}{\node[Nfill=y](n0)(0,0)}
\node[Nfill=y](ne)(40,-60){$(0,2)$}

\gasset{Nfill=n,Nframe=n,ExtNL=y,NLangle=90,NLdist=1}
\node(n1)(50,-80){$a$}
\node(n1)(60,-80){$a$}
\node(n1)(70,-80){$a$}
\node(n1)(80,-80){$a$}
\node(n1)(00,-70){$b$}
\node(n1)(10,-70){$a$}
\node(n1)(20,-70){$b$}
\node(n1)(30,-70){$a$}
\node(n1)(40,-70){$b$}
\node(n1)(50,-70){$a$}
\node(n1)(60,-70){$b$}
\node(n1)(70,-70){$a$}
\node(n1)(80,-70){$b$}
\node(n1)(00,-60){$b$}
\node(n1)(10,-60){$b$}
\node(n1)(20,-60){$b$}
\node(n1)(30,-60){$b$}
\node(n1)(40,-60){$b$}

}
\end{picture}
\caption{A word $w$ representing $[aaa\cdots)(\cdots abab\cdots)(\cdots bbb]$}
\label{figw}
\end{figure}


\begin{figure}[t,b,p]
  \centering
\begin{picture}(80,60)(10,-100)
\gasset{Nh=1.5,Nw=1.5,Nfill=n}
{\gasset{ExtNL=y,NLangle=-140,NLdist=0.3}
\multiput(0,-30)(0,-10){5}{
\multiput(0,0)(10,0){9}{\node(n0)(0,-20){}}}
\gasset{Nh=2,Nw=2,Nfill=y}

\node[Nfill=y](n0)(50,-80){$(1,0)$}
\drawline[AHnb=0](40,-50.0)(40,-90){}
\drawline[AHnb=0](0,-80)(80,-80){}

\multiput(50,-80)(10,0){4}{\node[Nfill=y](n0)(0,0)}
\multiput(0,-70)(10,0){9}{\node[Nfill=y](n0)(0,0)}
\multiput(0,-60)(10,0){5}{\node[Nfill=y](n0)(0,0)}
\node[Nfill=y](ne)(50,-60){$(1,2)$}

\gasset{Nfill=n,Nframe=n,ExtNL=y,NLangle=90,NLdist=1}
\node(n1)(50,-80){$a$}
\node(n1)(60,-80){$a$}
\node(n1)(70,-80){$a$}
\node(n1)(80,-80){$a$}
\node(n1)(00,-70){$a$}
\node(n1)(10,-70){$b$}
\node(n1)(20,-70){$a$}
\node(n1)(30,-70){$b$}
\node(n1)(40,-70){$a$}
\node(n1)(50,-70){$b$}
\node(n1)(60,-70){$a$}
\node(n1)(70,-70){$b$}
\node(n1)(80,-70){$a$}

\node(n1)(00,-60){$b$}
\node(n1)(10,-60){$b$}
\node(n1)(20,-60){$b$}
\node(n1)(30,-60){$b$}
\node(n1)(40,-60){$b$}
\node(n1)(50,-60){$b$}
}
\end{picture}
\caption{$aw = a[aaa\cdots)(\cdots abab\cdots)(\cdots
  bbb)$ and  $aw\neq wb$.}\label{figwa}
\end{figure}


\begin{figure}[t,b,p]
  \centering
\begin{picture}(80,60)(10,-100)
\gasset{Nh=1.5,Nw=1.5,Nfill=n}
{\gasset{ExtNL=y,NLangle=-140,NLdist=0.3}
\multiput(0,-30)(0,-10){5}{
\multiput(0,0)(10,0){9}{\node(n0)(0,-20){}}}
\gasset{Nh=2,Nw=2,Nfill=y}

\node[Nfill=y](n0)(50,-80){$(1,0)$}
\drawline[AHnb=0](40,-50.0)(40,-90){}
\drawline[AHnb=0](0,-80)(80,-80){}

\multiput(50,-80)(10,0){4}{\node[Nfill=y](n0)(0,0)}
\multiput(0,-70)(10,0){9}{\node[Nfill=y](n0)(0,0)}
\multiput(0,-60)(10,0){7}{\node[Nfill=y](n0)(0,0)}
\node[Nfill=y](ne)(60,-60){$(2,2)$}

\gasset{Nfill=y,Nframe=n,ExtNL=y,NLangle=90,NLdist=1}
\node(n1)(50,-80){$a$}
\node(n1)(60,-80){$a$}
\node(n1)(70,-80){$a$}
\node(n1)(80,-80){$a$}

\node(n1)(00,-70){$b$}
\node(n1)(10,-70){$a$}
\node(n1)(20,-70){$b$}
\node(n1)(30,-70){$a$}
\node(n1)(40,-70){$b$}
\node(n1)(50,-70){$a$}
\node(n1)(60,-70){$b$}
\node(n1)(70,-70){$a$}
\node(n1)(80,-70){$b$}

\node(n1)(00,-60){$b$}
\node(n1)(10,-60){$b$}
\node(n1)(20,-60){$b$}
\node(n1)(30,-60){$b$}
\node(n1)(40,-60){$b$}
\node(n1)(50,-60){$b$}
\node(n1)(60,-60){$b$}
}
\end{picture}
\caption{
$ aa [aaa\cdots)(\cdots abab\cdots)(\cdots bbb]
=  [aaa\cdots)(\cdots abab\cdots)(\cdots bbb]bb$}
  \label{figwaa}
\end{figure}



\begin{figure}[t,b,p]
  \centering
\begin{picture}(80,60)(10,-100)
\gasset{Nh=1.5,Nw=1.5,Nfill=n}
{\gasset{ExtNL=y,NLangle=-140,NLdist=0.3}
\multiput(0,-30)(0,-10){5}{
\multiput(0,0)(10,0){9}{\node(n0)(0,-20){}}}
\gasset{Nh=2,Nw=2,Nfill=y}

\node[Nfill=y](n0)(50,-80){$(1,0)$}
\drawline[AHnb=0](40,-50.0)(40,-90){}
\drawline[AHnb=0](0,-80)(80,-80){}

\multiput(50,-80)(10,0){4}{\node[Nfill=y](n0)(0,0)}
\multiput(0,-70)(10,0){4}{\node[Nfill=y](n0)(0,0)}
\node[Nfill=y](ne)(40,-70){$(0,1)$}

\gasset{Nfill=n,Nframe=n,ExtNL=y,NLangle=90,NLdist=1}
\node(n1)(50,-80){$x$}
\node(n1)(60,-80){$x$}
\node(n1)(70,-80){$x$}
\node(n1)(80,-80){$x$}

\node(n1)(00,-70){$y$}
\node(n1)(10,-70){$y$}
\node(n1)(20,-70){$y$}
\node(n1)(30,-70){$y$}
\node(n1)(40,-70){$y$}

}
\end{picture}
\caption{An infinite word  $t = [xxx\cdots)(\cdots
  yyy]$ where  $xt =  ty$.}
  \label{figxw}
\end{figure}

If $G$ is an infinite  group, then there are \Gred $A$-words of arbitrary length. 

\begin{lemma}\label{lem:totallang}
 Let $G$ be an infinite  group and $\alpha \in A$. Then there exists a \Gred $A$-word 
$x: [1,\alpha] \to \GG$ of length $\alpha$.
\end{lemma}
\begin{proof} First, let us assume that
$\GG$ is finite. We may assume that letters of $\GG$ are \Gred.
  There are infinitely many finite \Gred words in $\GG^*$,
  simply because each group element can be represented this way. They
  form a tree in the following way. The root is the empty word $1$. A
  letter has $1$ as its parent node.  A finite \Gred word of the
  form $w=avb$ with $a,b \in \GG$ has $v$ as its parent node.
  Since $\GG$ is finite the degree of each node is finite. Hence
  K\"onig's Lemma tells us that  
  there must be an infinite path. Following this path from the root
  yields a partial word $p: \Z \rightarrow \GG$ in an obvious way:
  If $v$ denotes the \Gred word $v:[m,n] \to \GG$, then $w=avb$
  denotes the \Gred word $w:[m-1,n+1] \to \GG$ 
  with $w[m-1] = a$, $w[m,n] = v$, and $w[n+1] = b$.  This mapping
  $p: \Z \rightarrow \GG$ can be extended to a mapping $q: A
  \rightarrow \GG$ by $q(\sum_i n_it_i) = p(n_0)$. This means we
  project $\alpha \in A$ to the first component and then we use $p$.
  For every $\alpha \in A$ the partial word $q[1,\alpha] \rightarrow
  \GG$ is \Gred.

  If $G$ is finitely generated but $\GG$ is infinite then one can repeat the argument above for some large enough finite subset of $\GG$ (that generates $G$). It remains to consider the case that $G$ is not finitely generated. Assume 
  a \Gred word $v:[m,n] \to \GG$ has been constructed. Then we choose
  $a \in \GG$ such that $a$ is not in subgroup generated by the elements
  $v[i]$ for $m \leq i \leq n$. Clearly, $av: [m-1,n] \to \GG$ is \Gred. 
  Next choose  $b \in \GG$ such that $b$ is not in subgroup generated by the elements
  $av[i]$ for $m-1 \leq i \leq n$. 
  Now,  $avb: [m-1,n+1] \to \GG$ is \Gred. 
  We obtain a \Gred word  $p: \Z \rightarrow \GG$ and we argue as above.   
\end{proof}

By $R(A,G)$ we denote the set of all \Gred words in $W(A,\GG)$, and by $R^*(A,G)$ we mean the submonoid of 
$W(A,\GG)$ which is generated by $R(A,G)$. 

\begin{remark}\label{rem:uninteressant} In the notation above:
\begin{itemize}
\item If the group $G$ is finite, then $R(A,G)$
 cannot contain any infinite word, and in this case $R^*(A,G) = \GG^*$. 
 \item If $A = \Z$ then $W(\Z,\GG) = \GG^*$. 
 \end{itemize}
 These  situations are without any interest in our context, 
 so we assume in the sequel that $G$ is infinite and that $A$ has rank at least 2 (i.e., it is \nonarch). 
\end{remark}

Observe, that the length function $W(A,\GG)\to A, p \mapsto \abs{p}$ induces a canonical 
\homo onto $\oplus_{i \in \Omega} \Z / 2\Z$ which therefore factors through 
the greatest quotient group of $W(A,\GG)$. This group 
collapses $\SS$ into a group of order 2, and therefore the greatest quotient group of $W(A,\GG)$ is of no particular interest here. More precisely, 
we have the following fact. 

\begin{proposition}\label{collapse}
Let  $\SS \neq  \es$ and $$\psi: F(\SS)\to  W(A,\GG)/ \set {u\ov u = 1}{u \in W(A,\GG)}$$ be the canonical 
\homo induced by $\SS \sse W(A,\GG)$, and let $A$ have  rank at least 2. Then the image of  
$F(\SS)$  under $\psi$ is the group $\Z / 2\Z$.
\end{proposition}

\begin{proof}
The image of $F(\SS)$ is not trivial, because it is non-trivial in the group 
$\oplus_{i \in \Omega} \Z / 2\Z$. 
It is therefore enough to show that $\psi(ab) = 1$ for all
$a,b \in \GG$. Consider the following closed word $u$  of length $(0,1)$:
 $$u = [ababab\cdots\;)(\;\cdots a\ov{a}a\ov{a}a\ov{a}]$$
In $W(A,\GG)$ we have $abu=ua\ov{a}$.
Now,  $\psi(a\ov{a})=1$ implies    $\psi(ab)=1$.
\end{proof}

Continuing with $F(\SS)$, consider 
 the following word $w$ 
 of length $(0,2)$, which is product of two freely reduced  words where $a,b \in \GG$ with
 $a \not= \ov b$:
$$w=[aaa\cdots\;)(\;\cdots aaa ]\cdot[\ov{a}\ov{a}\ov{a}\cdots\;)(\;\cdots bbb]$$

It is natural to allow (and we will do) the cancellation of factors $a\ov{a}$ inside $w$. 
The shape of the word remains the same, but the length
is decreasing to any value $(-2n,2)$ with $n \in \N$.
If next we wish to embed $F(\SS)$ into any quotient structure of  $W(A,\GG)$,
then we cannot cancel however the whole middle
part $(\;\cdots aaa ]\cdot[\ov{a}\ov{a}\ov{a}\cdots\;)]$, i.e.,  $w$ cannot become equal to  $v= [aaa\cdots\;)(\;\cdots bbb]$
in this quotient. Indeed, assume by contradiction  $w = v$, then: 
  \[
\begin{array}{rcl}
aa v = avb &=& a[aaa\cdots\;)(\;\cdots bbb]\,b\\
&=& a\,w\,b\\
&=&[aaa\cdots\;)(\;\cdots aaa ]\cdot a \ov{a}\cdot [\ov{a}\ov{a}\ov{a}\cdots\;)(\;\cdots bbb]\\
&=& w= v.
\end{array}
\]
This implies $a^2= 1$, a contradiction.

\section{The group $\Ext(A,G)$}
Proposition~\ref{collapse} shows that, in general,  the free group $F(\SS)$ does not naturally embed into the greatest quotient group of  
 $W(A,\GG)$.
Nevertheless, in this section we modify the construction to be able to represent a group $G$ by infinite words from $W(A,\GG)$. As above, 
we let $G$ be a group generated by  $\Sigma$ and $\pi:\GG^*\to G$ be the 
induced presentation with $\GG= \SS\cup \SS^{-1}$. 
Recall that
 $R(A,G)$ denotes the set of closed \Gred words, 
 i.e.:  $$R(A,G)= \set{u \in W(A,\GG)}{u \mbox{ is } \mbox{\Gred}}.$$
  Let $\MM(A,G)$ be the following quotient 
monoid of $W(A,\GG)$:
\begin{equation*}
 \MM(A,G) = W(A,\GG)/\set{u\ell\ovv{r}{u}=1}{u \in R(A,G), \ell,r \in \GG^*, \pi(\ell)=\pi(r)}.
\end{equation*}

\begin{definition}
We define  
$\Ext(A,G)$ as the image of $R^*(A,G)$ in $\MM(A,G)$ under the canonical  epimorphism $W(A,\GG) \to \MM(A,G) $.
\end{definition}

In the following proposition we collect some simple results on  $\Ext(A,G)$. 

\begin{proposition}
Let $G$ be a group generated by a set  $\Sigma$ and $A = \oplus_{i \in \Omega} \gen{t_i}$ as above. Then:

\begin{enumerate}[1)]
\item $\Ext(A,G)$ is  a group  (a subgroup of $\MM(A,G) $);
\item every submonoid of  $\MM(A,G)$ which is a group sits inside the group $\Ext(A,G)$, so $\Ext(A,G)$ is the group of units  in $\MM(A,G)$;
\item the 
inclusion $\GG \sse G$ induces 
a homomorphism  $\pi^A: G\to \Ext(A,G)$. 
\end{enumerate}
\end{proposition}
\begin{proof}
To see 1) observe that  every element in $u \in R(A,G)$ has $\ov u$ as an inverse in $\Ext(A,G)$, so  $\Ext(A,G)$ is a group. 

Notice that only the trivial word is invertible in $W(A,\GG)$ since concatenation does not decrease the length. Hence every  equality $w\ov{w}  = 1$ for a non-trivial $w$ in $W(A,\GG)$ comes from the defining relations in $\MM(A,G)$.  Observe, that the defining relations are applicable only to words from  $R^*(A,G)$, the set $R^*(A,G)$ is closed under such transformations.  This shows that  $\Ext(A,G)$ is the group of units  in $\MM(A,G)$, as claimed in 2).
 
3) is obvious since  $G = \GG^*/\set{\ell\ov{r}=1}{\pi(\ell)=\pi(r)}$
and $1\in R(A,G)$. 
\end{proof}
 
 Several important remarks are due here. 
 
 \begin{itemize}
 
 \item It is far from 
obvious that the homomorphism $\pi^A: G\to \Ext(A,G)$ is injective.  However, this is true and  we  prove it  later in \refcor{cor:embedG}.

\item  It is not claimed that the definition of $\MM(A,G)$ (or $\Ext(A,G)$) 
is independent of the choice of $\GG$ and $\pi$, but our main results 
hold through for any such  $\GG$ and $\pi$ thus justifying   the 
(sloppy) notations $\MM(A,G)$ and $\Ext(A,G)$.
%
%
  \item 
If $G = F(\SS) $ is the free group with basis $\SS$, then the definition
of $\MM(A,G)$ can be rephrased by saying that it 
is the quotient
monoid of $W(A,\GG)$ with defining equations $u \ov u = 1$ for all freely reduced 
closed words $u$. 

\item
It is not true in general 
that $\Ext(A,G)$ can be defined as the 
quotient group 
$$\Ext(A,F(\SS))/ \set{\ell=r}{ \ell,r \in \GG^*, \pi(\ell)=\pi(r)}.$$ Indeed, let $r$ be a cyclically reduced word of length $m$ such that 
$r= 1$ in $G$. In $\Ext(A,F(\SS))$ for every $a \in \GG$ the words $a^m$ and $r$ are conjugated
since 
$$a^m  [a^ma^ma^m\cdots\;)(\;\cdots rrr] = [a^ma^ma^m\cdots\;)(\; \cdots rrr] r.$$
Therefore, $a^m = 1$ in $\Ext(A,F(\SS))/ \set{\ell=r}{ \ell,r \in \GG^*, \pi(\ell)=\pi(r)}$, which may not be the case in $G$ (which is a subgroup of $\Ext(A,G)$).

\end{itemize}
 Nevertheless,  $\Ext(A,F(\SS))$ satisfies some universal property. 
 
 \begin{proposition}
 Every group $G$ generated by $\SS$ is isomorphic to the  canonical  quotient  of the subgroup in $\Ext(A,F(\SS))$ generated by $R(A,G)$. 
\end{proposition}
\begin{proof}
The statement is obvious. 
\end{proof}

\section{Confluent rewriting systems over \nonarch words}
\label{sec:csoverdoag}

Our goal here is to construct a 
confluent rewriting system $S$ over the monoid  $W(A,\GG)$ such that  
$$\MM(A,G) =  W(A,\GG)/S$$
and $S$ has the following form: 
\begin{equation}\label{eq:hugo}
 S = S_0 \cup \set{u\overline{u}\rightarrow 1}{u \in R(A,G) \text{ and } 
 u \text{ is infinite}},
\end{equation}
where  $S_0 \subseteq \GG^* \times \GG^*$ is a rewriting system for $G$ satisfying the following conditions:
\begin{enumerate}
 \item $\GG^* / S_0 = G$
 \item For all $a \in \GG$ we have $(a\overline{a}, \, 1) \in S_0$. 
 \item If $(\ell,r) \in S_0$,  then $(\ov \ell,\ov r) \in S_0$. 
 \item $1 \in \GG^*$ is $S_0$-irreducible.
 \item $S_0$ is confluent.
\end{enumerate}
In general, $S_0$ is neither finite nor terminating, but these conditions are not crucial for the moment, so we do not care. 

\begin{lemma}
For any group $G$ generated by $\SS$ there is a rewriting system $S_0\subseteq \GG^* \times \GG^*$ satisfying the conditions 1-5 above.  Moreover, 
if $G$ is finitely presented, then one can choose $S_0$ to be finite.
\end{lemma}

\begin{proof}
Let $G = \GG^*/R$ for some set of defining relation $R$. In general, let $S_0$ be the set of all rules $u \ra v$, where 
$u$ is non-empty and  and $u \neq v $ as words, but $u = v $ in $G$. Notice that there are no rules $1\ra r$ in $S_0$, so $1  \in IRR(S_0)$. However, for every $r \in R\cup \ov{R}$ and every letter $a \in \SS$ the relations $a \ra ra$ and $a\ra ar$ are in $S_0$, so one can insert any relation $r$ in a word, thus simulating the rule $1 \ra r$.  

In the case when $R$ is finite consider only those rules $u \ra v$ from $S_0$ such that   $\abs{u}+\abs{v} \leq k+2$, where $k = \max\{\abs{\ell}+\abs{r}\mid \ell \ra r \in R\}$. Notice, again that all the rules of the type  $a \ra ra$ and $a\ra ar$ are in $S_0$.
\end{proof}

Clearly:
\begin{equation*}
 M(A,\GG) = W(A,\GG) / S.
\end{equation*}

The following lemma will be used only later.  The proof shows however our basic
techniques to factorize and to reason about rewriting steps. The reader is therefore invited to read the proof carefully. 

\begin{lemma}\label{lem:otto}
Let $x \in R(A,G)$ be a non-empty \Gred word. Then 
$x\RAS *S y$ implies both $x\RAS *{S_0} y$ and $y$ is a non-empty word. 
\end{lemma}

\begin{proof}
By contradiction, assume $x\RAS *S y$, but not  $x\RAS *{S_0} y$.  
Then there are an  infinite \Gred word $u\in R(A,G)$ and some closed word
$y_0$  such that $x\RAS *{S_0} y_0 \RAS *S y$
where the rule $u\ov u \ra 1$ applies to $y_0$. Note that 
rules of $S_0$ replace  left-hand sides inside finite intervals. These intervals can be made larger and if two of them are separated by a finite distance, then  we can join them. Hence we obtain a picture as follows where all $x_i$ are infinite, and
all $f_i$, $g_i$  are finite words:
\begin{align*}
x &= x_1f_1 \cdots x_{n-1}f_{n–1}x_n \\
y_0 &= x_1g_1 \cdots x_{n-1}g_{n–1}x_n = pu \ov uq, \\
pq &\RAS*S y, \\
f_i &\RAS*{S_0}g_i \mbox{ for }  1 \leq i \leq n. 
\end{align*}
The middle position of $y_0 = pu \ov uq$ between $u \ov u$ cannot be  inside some factor $x_m$ as $x$ is \Gred. The middle position meets therefore some finite factor $g_m$. Thus,  (as 
$u$ is  infinite) we can enlarge $f_m$ such that $f_m \RAS*{S_0}1$. 
This implies $f_m = g_m = 1$ as words, because $x$ is \Gred and $1$ is irreducible w.r.t.{} 
$S_0$. Let $a$ be the last letter of $u$, then it is the last letter of $x_m$
and $\ov a$ is the first letter of $x_{m+1}$, too. Hence $a \ov a$ appears as 
a factor in $x$. This is a contradiction, and therefore $x\RAS *{S_0} y$.

 Since  $x$ is a non-empty  \Gred word, we cannot 
 have  both $x\RAS *{S_0} y$ and $y = 1$. 
  \end{proof}

The main technical result of this section is the following theorem.
\begin{theorem}\label{thm:confl}
 The system $S \subseteq  \GG^* \times \GG^* \cup R(A,G) \times R(A,G)$ defined in Equation~\ref{eq:hugo} is
confluent on $W(A,\GG)$.
\end{theorem}

For technical reasons we replace the rewrite system $\RAS{}{S}$
by a new system which is denoted by  $\RAS{}{\BG}$. 
It is defined by 
$$\RAS{}{\BG} \ \  = \  \ \RAS{*}{S_0}\circ \RAS{}{S}\circ \RAS{*}{S_0}.$$
We have $x \RAS{}{\BG} y$ \IFF{} there is a derivation  
$x \RAS{+}{S} y$ which may use many times rules {}from $S_0$, but at most once a  rule 
 {}from the sub system $$\set{u\overline{u}\rightarrow 1}{u \in R(A,G) \text{ and } 
 u \text{ is infinite}}.$$ The notation is due to the fact that 
 we can think of \emph{Big} rules in this subsystem.

The proof of \refthm{thm:confl} is an easy consequence
of the following lemma. However, the proof of this lemma is somehow tedious,  technical, and rather long. 
\begin{lemma}\label{lem:confllem}
 The rewriting system $\RAS{}{\BG}$ is
strongly confluent on $W(A,\GG)$.
\end{lemma}

\begin{proof}
We start with the situation 
$$\OUTS{y}{}{\BG}{x}{z},$$
and we have to show that there is some $w$ with 
$$y \RAS{\leq 1}{\BG} w \LAS{\;\;  \leq 1}{\BG} z.$$

This is clear, if we have $\OUTS{y}{*}{S_0}{x}{z},$
because $S_0$ is confluent and several steps using $\RAS{}{S_0}$
yield at most one step in $\RAS{}{\BG}$.

Next, we consider  the following situation 
$$y \LAS{*}{S_0} y_1 \LAS{}{S} y_0 \LAS{*}{S_0} x \RAS{*}{S_0}z.$$
We content to find a $w$ such that 
 $$y_1 \RAS{*}{S_0} w \LAS{\;\;  \leq 1}{\BG}z.$$
Here comes a crucial observation which is used throughout in the
following (compare to the proof of \reflem{lem:otto}). We find factorizations as follows.
\begin{align*}
x &= f_0 x_1f_1 \cdots x_nf_n \\
y_0 &= g_0 x_1g_1 \cdots x_ng_n \\
z &= h_0 x_1h_1 \cdots x_nh_n
\end{align*}
Moreover, all $f_i$ are finite, all $x_i$ are infinite,  and always:
$$\OUTS{g_i}{*}{S_0}{f_i}{h_i}.$$ 
In addition we may assume that $y_0= p u \ov u q$ with 
 $y_1= p q$ and $u$ is an infinite \Gred word. We can shrink
 $u$ by some finite amount and we can make all
 $f_i$ larger and we can split some $x_i$ into factors. As a consequence we may assume 
 the left-hand side $ u \ov u $ covers exactly some
 factor $x_\ell \cdots x_k$ for  
 $1\leq \ell\leq k\leq n$. In particular, we have 
 $$y_1 = g_0 x_1g_1 \cdots x_{\ell-1}g_{\ell-1} 
 g_{k+1}  x_{k+1}x_ng_n.$$
 Since $S_0$ is confluent, it is enough to consider the case  
 $x = x_\ell \cdots x_k$. 
 We may therefore simplify the  notation and 
 we assume the following:
\begin{align*}
x &= x_1f_1 \cdots x_{n-1}f_{n–1}x_n \\
y_0 &= x_1g_1 \cdots x_{n-1}g_{n–1}x_n = u \ov u \\
z &= x_1h_1 \cdots x_{n-1}h_{n–1}x_n \\
y_1 &=1
\end{align*}
We may assume that the middle position between $u$ and $ \ov u$
is inside some factor $g_m$. By making $f_m$ larger we may assume
that $g_m$ has the form $g_m= r_m \ov{r_m}$. But then 
we have $h_m \RAS{*}{S_0} 1$, and hence we may assume
that 
$f_m = g_m = h_m = 1$. 
Refining the partition, making $f_i$ larger, and shrinking $u$ by some 
finite amount, we arrive at the following situation
with $n =2m$ and
\begin{align*}
y_0 &= x_1g_1 \cdots x_{m-1}g_{m–1}x_m 
\;\ov{x_m}\;   \ov{g_{m–1}}\;  \ov{x_{m-1}} \cdots \ov{g_1} \; \ov{x_1} 
\end{align*}

As $u = x_1g_1 \cdots x_{m-1}g_{m–1}x_m$ we see that all $x_i$ are \Gred. 
For each $1\leq i \leq m-1$ we find $r_i$ such that
$g_i \RAS{*}{S_0} r_i$, $ h_i \RAS{*}{S_0} r_i$, and 
$ h_{m+i} \RAS{*}{S_0} \ov{r_{m-i}}$.
As a consequence we may assume 
\begin{align*}
z &= x_1r_1 \cdots x_{m-1}r_{m–1}x_m 
\;\ov{x_m}\;   \ov{r_{m–1}}\;  \ov{x_{m-1}} \cdots \ov{r_1} \; \ov{x_1} 
\end{align*}
Note that it is not clear that 
the word  $x_1r_1 \cdots x_{m-1}r_{m–1}x_m$ is \Gred. 
So we start looking for a finite non-empty factor $h$ with 
$ h \RAS{*}{S_0} 1$. If we find such a factor, we cancel it 
and we cancel the corresponding symmetric factor $\ov h$
on the right side in $\ov{x_m}\;   \ov{r_{m–1}}\;  \ov{x_{m-1}} \cdots \ov{r_1} \; \ov{x_1}$. The factor must use a piece of
some $r_i$ because all $x_i$ are \Gred. 
But it never can use all of some  $r_i$ because
$x_1g_1 \cdots x_{m-1}g_{m–1}x_m$ is \Gred. Thus, the cancellation process stops 
and we can replace $z$ by some word which has the 
form $z= v \ov v$, where $v$ is indeed \Gred. 
Thus, the rewrite step $ z \RAS{}{\BG} 1$ finishes the 
situation $$y \LAS{*}{S_0} y_1 \LAS{}{S} y_0 \LAS{*}{S_0} x \RAS{*}{S_0}z.$$

For later later use we recall that we found some $w$ and derivation
as follows: 
$$y \RAS{*}{S_0} w \LAS{\;\;\leq 1}{ \BG}z .$$

The challenge is now to consider a situation as follows. 
$$y \LAS{*}{S_0} y_1 \LAS{}{S} y_0 \LAS{*}{S_0} x \RAS{*}{S_0}z_0
 \RAS{}{S} z_1 \RAS{*}{S_0}z.$$

We claim that it is enough to find some $w$ with 
$$y_1 \RAS{\leq 1}{\BG} w \LAS{\;\;\leq 1}{ \BG}z_1 .$$

Indeed, if such a $w$ exists, then we have just seen that there are  $w_1$, $w_2$ with
$$y \RAS{\leq 1}{\BG} w_1  \LAS{*}{S_0}w
 \RAS{*}{S_0} w_2 \LAS{\;\;\leq 1}{ \BG}z .$$
 By confluence of $S_0$ there is 
  some $w'$ with
 $$ w_1  \RAS{*}{S_0} w' \LAS{*}{S_0} w_2.$$
 We are done, because now
 $$ y  \RAS{\leq 1}{\BG} w' \LAS{\leq 1}{\BG} z.$$
The claim now implies that we are left with the following case:
$$y  \LAS{}{S} y_0 \LAS{*}{S_0} x \RAS{*}{S_0}z_0
 \RAS{}{S} z.$$

We repeat the assumptions and notations {}from above.
We have \begin{align*}
x &= f_0 x_1f_1 \cdots x_nf_n \\
y_0 &= g_0 x_1g_1 \cdots x_ng_n \\
z_0&= h_0 x_1h_1 \cdots x_nh_n
\end{align*}
All $f_i$ are finite, all $x_i$ are infinite,  and always:
$$\OUTS{g_i}{*}{S_0}{f_i}{h_i}.$$ 
We may assume that $y_0= P u \ov u q$ 
and $z_0= p v \ov v Q$  with 
 $y_1= P q$ and $y_1= pQ $ and $u$ and $v$ are  infinite \Gred words. We can shrink
 $u$ and $v$ by some finite amount and we can make all
 $f_i$ larger and we can split some $x_i$. As a consequence we may assume 
 the left-hand side $ u \ov u $ covers exactly some 
 factor $x_\ell \cdots x_k$ with  
 $1\leq \ell\leq k\leq n$, and 
 the left-hand side $ v \ov v $ covers exactly some
 factor $x_L \cdots x_K$ with 
 $1\leq L \leq K\leq n$.
 We say that $g_i$ is covered by $ u \ov u $, if 
 $\ell\leq i < k $. If $g_i$ is not covered, then we may assume that
 $g_i = f_i$. Analogously, $h_i$ is covered by $ v \ov v $, if 
 $L \leq i < K $. If $h_i$ is not covered, then we may assume that
 $h_i = f_i$. 
 
 We may assume that $\ell \leq L$. If there is no overlap
 between the factors  $ u \ov u $ and  $ v \ov v $, i.e., 
 if $k < L$, then the situation is trivial, because those 
 $g_i$ or $h_i$ which are not covered, are still equal to $f_i$. 
 Thus, we have overlap. Moreover, we may assume that 
 $f_0= f_n = 1$, $\ell = 1$, and $n = \max\smallset{k,K}$. 
 In order to clarify we repeat
 \begin{align*}
x &=  x_1f_1 \cdots x_n \\
y_0 &= x_1g_1 \cdots x_n = u \ov u q\\
z_0&= x_1h_1 \cdots x_n = p v \ov v Q \text{ and either } q= 1 \text{ or } Q=1\\
y &= x_{k+1}f_{k+1} \cdots x_{n-1}f_{n-1}x_n \\
z&= x_1f_1 \cdots x_{L-1}f_{L-1}\, f_{K}x_{K+1}  \cdots f_{n-1}x_n 
\end{align*}
 
 We are coming to a subtle point. As above we may assume that
 the middle position between $u \ov u$ is inside some $g_m$ and
 and that the middle position between $v \ov v$ is inside some $h_M$.
 There are two cases $m = M$ or   $m \neq M$. 
 Let us treat the case  $m = M$, first. 
 
 Given the preference 
 to $u$ we may enlarge $f_m$ such that $g_m= r \ov{r}$. 
 Thus, actually we may assume $g_m=1$. 
 However it is not clear that $h_m$ can be factorized the same way. 
 But $h_m$ is finite and $v$ is infinite, hence, by left-right
 symmetry, we have $h_m = s \ov s h$, where  $\ov s h$
 is a prefix of $\ov v$. Now, in the group $G$ we have 
 $1=g_m=f_m=h_m=h$. Since $h$ is a factor of $\ov v$ and $\ov v$ 
 is \Gred, we conclude that $h= 1$ as a word. This allows to conclude 
 that $f_m= g_m = h_m =1 $ as words. 
  Again, by left-right
 symmetry, we may assume that $\ov x_m$ is a prefix
 of $x_{m+1}$. Thus, both in $y_0$ and in $z_0$ we replace the
 common factors $x_m \ov x_m$ by $1$. Note that this has no influence on $y$ or $z$.  This yields a new assumption 
 about $x$, $y_0$, and $z_0$, we have 
 \begin{align*}
x &=  x_1f_1 \cdots x_{n'}  
\end{align*}
with $n'\leq n$ and a corresponding $m' = M'<m$.
We repeat the procedure. There is only one way the procedure may stop. 
Namely at some point $v$ is not an infinite factor anymore.

Hence, we are back at a situation of type:
$$y \LAS{*}{S_0} y_1 \LAS{}{S} y_0 \LAS{*}{S_0} x \RAS{*}{S_0}z.$$
This situation has already  been solved. 

Hence for the rest of this proof we may assume $m \neq M$. 
This is actually the most 
difficult
part. By making $f_m$ and $f_M$ larger, we may assume
that $g_m = h_M= 1$ as words. Note that for some letter
$a$ we have $x_m = x'a$ and $x_{m+1} = \ov a x''$
Assume that $h_m$ is covered 
by $v \ov v$. Then  $ah_m\ov a$ appears as a non-trivial 
factor in $v \ov v$, where $ah_m\ov a \RAS{*}{S_0}1$. 
Since both $v$ and $ \ov v$ are \Gred, we end up with 
$m=M$, which has been excluded. Thus, $h_m$ is not covered 
by $v \ov v$. We conclude that we may assume $f_m=g_m=h_m= 1$ as words. 
By symmetry, $g_M$ is not covered 
by $u \ov u$ and $f_M=g_M=h_M= 1$ as words. 
In particular we have $k \leq K$. More precisely, we
are  now faced with the following situation: 
$$1= \ell \leq m \leq L \leq k \leq M \leq K=n.$$
Without restriction we can therefore write:
\begin{align*}
x &=  x_1f_1 \cdots x_L f_L \cdots f_{k-1}x_k \cdots f_{n-1}x_n \\
y_0 &= x_1g_1 \cdots x_{k-1}  g_{k-1}x_k  f_{k}x_{k+1} \cdots f_{n-1}x_n = u \ov u y\\
z_0&= x_1f_1 \cdots x_{L-1}  f_{L-1}x_{L}h_L x_{L+1}\cdots h_{n-1}x_n
=  z v \ov v\\
y &= f_{k}x_{k+1} \cdots f_{n-1}x_n \\
z&= x_1f_1 \cdots x_{L-1}f_{L-1}
\end{align*}
 
Consider the \emph{overlapping} factor  $ \widetilde x= x_L f_L \cdots f_{k-1}x_k$
inside the word $x$.
Define new  words $w_g = x_L g_L \cdots g_{k-1}x_k$
and $w_h = x_L h_L \cdots h_{k-1}x_k$.
We claim that there are  \Gred words $U$ and $V$ such that 
\begin{align*}
y & \RAS{*}{S_0}  V \ov V \ov{w_h},\\
z & \RAS{*}{S_0}  \ov{w_g} U \ov U.\\
\end{align*}

By symmetry it is enough to show that $y  \RAS{*}{S_0}  V \ov V \ov{w_h}$. 
Consider 
$$y_0 = x_1g_1 \cdots x_{k-1}  g_{k-1}x_k  f_{k}x_{k+1} \cdots f_{n-1}x_n.$$
Since $f_M=1$ we know that $x_{L}f_L x_{L+1}\cdots f_{M-1}x_M$ reduces 
to the word $v$ and hence $x_{M+1}f_{M+1}\cdots f_{n-1}x_n$ reduces to $\ov v$.
Moreover, we can write $v = w_h V$ with  $$f_k x_{k+1}\cdots f_{M-1}x_M \RAS{*}{S_0} V.$$
As $V$ appears in a factor of $v$ it is \Gred. We obtain the claim:  
$$y= f_k x_{k+1}\cdots f_{M-1}x_M x_{M+1}f_{M+1}\cdots f_{n-1}x_n \RAS{*}{S_0} V \ov V \ov{w_h}.$$

Since $S_0$ is confluent and $w_g \LAS{*}{S_0} \widetilde x \RAS{*}{S_0} w_h$, we find  $w$
such that 
$$\ov{w_h}  \RAS{*}{S_0} w  \LAS{*}{S_0} \ov{w_g}$$
Hence: 
$$y \RAS{\leq 1}{\BG} w \LAS{\,\, \leq 1}{\BG} z.$$
 
This shows that the system $S$ in Equation~\ref{eq:hugo}
is confluent. This finishes the proof of the lemma and therefore of Theorem~\ref{thm:confl}, too.
\end{proof}

\begin{corollary}\label{cor:embedG}
The canonical homomorphism $G\to \Ext(A,G)$ is an embedding.
\end{corollary}

\begin{proof}
 Let $x,y\in \GG^*$ be finite words such that $x=y$ in $\Ext(A,G)$.
 Then we have $\INS x*Swy$ for some $w\in \GG^*$. But this implies 
 $\INS x*{S_0}wy$. Hence $x=y$ in $G$.
\end{proof}
 
 \begin{corollary}\label{covergentS}
Let $S_0$ be a convergent system defining the group $G$. 
The canonical mapping  $$\IRR(S_0) \cap R(A,G) \to \Ext(A,G)$$ is injective.
\end{corollary}

\begin{proof}
 Since the system $S$ is confluent (hence Church-Rosser), the canonical mapping  $\IRR(S)  \to \Ext(A,G)$ is injective.
 The result follows, because  Lemma~\ref{lem:otto} tells us $\IRR(S) \cap R(A,G) = \IRR(S_0) \cap R(A,G)$. 
 \end{proof}

The following special case is used in \refsec{sec:torsion}. 
 \begin{corollary}\label{freeredS}
Let  $G = F(\SS)$ be a free group. Then pairwise different freely reduced closed 
words are mapped to pairwise different elements in 
$\Ext(A,G)$.
\end{corollary}

\begin{proof}
 For $G = F(\SS)$ we can choose $S_0$ to contain just the trivial rules
 $a\ov a \ra 1$, where $a \in \GG = \SS \cup \SS^{-1}$. The system 
 is convergent and $$\IRR(S_0) =  R(A,G)= 
 \set {u \in W(A,G)}{u \mbox{ is freely reduced}}.$$
 The result follows by Corollary~\ref{covergentS}.
 \end{proof}


\begin{example}\label{ex:conj}
Let $a\in \SS$ and
$u,v \in F(\SS)$ be represented
by non-empty cyclically reduced words in $\GG^*$. (For example $u,v$ are themselves letters.) Consider the following infinite  closed words:
\begin{align*}
w=[uuu\cdots\;)&(\;\cdots vvv]\\
z=[uuu\cdots\;)(\;\cdots aaa ]&[\ov{a}\ov{a}\ov{a}\cdots\;)(\;\cdots \ov{v}\ov{v}\ov{v}]
\end{align*}
The word $w$ is freely reduced, hence irreducible w.r.t.{} the system $S_0$. The word
$z$ is not freely reduced and $S_0$ is not terminating on $z$.

By \refcor{covergentS} we have  $uw = wv$ in $\Ext(A,G)$ \IFF $uw = wv$ in $W(A,G)$ $\abs{u} = \abs{v}$. 

Although the word $z$ has no well-defined 
length one can infer the same conclusion. First let $\abs{u} = \abs{v}$, then 
$z = u z \ov v$ in $\Ext(A,G)$ and hence $uz = zv$. For the other direction 
write $z = z' \ov v$ as words and 
let $uz = zv = z'$ in $\Ext(A,G)$. Then $uz \RAS{*}{S} \widetilde z \LAS{*}{S} z'$
for some word $\widetilde z$.  

After cancellation 
of factors $a^m\ov a ^m$ inside  $(\;\cdots aaa ]\cdot[\ov{a}\ov{a}\ov{a}\cdots\;)$ the borderline between $a$'s and $\ov a$'s must match inside $\widetilde z$. So exactly
$\abs{u}$ more cancellations of type $a\ov a\ra 1$ inside $uz$ took place than 
in $z'$. Hence $\abs u = \abs v$. The other direction is trivial. 
\end{example}

 For each ordinal $d \in \Omega$ let 
 $$\G_d= \set{x \in \Ext(A,G)}{x \,\mbox{ is given by some word of degree at most } \, d}$$
 Corollary~\ref{cor:embedG} has an obvious generalization. The proof is by transfinite induction and left to the interested reader. 
 
\begin{corollary}\label{cor:embedGd}
Let $d\leq e  \in \Omega$.
Then the canonical homomorphism $\G_d \to \G_e$ is an embedding.
\end{corollary}

The group $\Ext(A,G)$ is the union of all $\G_d$, but if $G$ is finite
nothing interesting happens, we have $G = \Ext(A,G)$ in this case because
there are no infinite \Gred words. However if $G$ is infinite, then 
$\Ext(A,G)$ may become huge due to the following observation. 

\begin{proposition}\label{prop:huge}
Let 
$A$ have rank at least 2.
Then the  following assertions are equivalent:
\begin{enumerate}[i.)]
\item The group $G$ is infinite. 
\item  For all $d < e  \in \Omega$ we have $\G_d \neq \G_e$.
\end{enumerate}
\end{proposition}
\begin{proof}
We have  $\abs{\Omega} \geq 2$.
Let $d < e  \in \Omega$ with $\G_d=  \G_e$. We show that
$G$ is finite. Assume the contrary, then by Lemma~\ref{lem:totallang}
there is some \Gred word $x$ of degree $e$. Assume we find a word 
$z$ of degree at most $d$ such that $x \DAS *S z$. Then, be confluence of $S$
we have $\INS x*S yz$
for some $y$ of degree at most $d$. But now Lemma~\ref{lem:otto} tells us that 
$x \RAS *{S_0} y$, which implies that $x$ is of degree $d$, too. 
This is a contradiction, because rules from $S_0$ cannot decrease any degree other than 0. 
\end{proof}

The notion  of a pre-perfect system  from Definition~\ref{def:preperfect}
can be applied to rewriting systems over $W(A,\GG)$, too. In this case Theorem~\ref{thm:confl} implies the following result.

\begin{corollary}\label{rem:preperfecttrans}
If the group $G$ 
is defined by some pre-perfect string rewriting system $S_0$, then 
 the system $S$ on $W(A,\GG)$ is also pre-perfect.
\end{corollary}

\begin{definition}\label{def:short}
A word  $x \in W(A,\GG)$ is called a  \emph{local geodesic}, if it has  no finite factor $f$ such that $f=g$ in $G$ and $\abs{g}<\abs{f}$.
\end{definition}

\begin{proposition}\label{prop:shortest}
Let $G$ be presented by some pre-perfect string rewriting system $S_0\subseteq \GG^* \times \GG^*$. 
Let $x \in W(A,\GG)$ be a local geodesic. Then 
$x\RAS *S y$ implies both $x\RAS *{S_0} y$ and $\abs{x}= \abs{y}$. 
\end{proposition}

\begin{proof}
Straightforward {}from Lemma~\ref{lem:otto} since local geodesics are \Gred. 
 \end{proof}

\section{Torsion elements in $\Ext(A,G)$ and cyclic decompositions}\label{sec:torsion}
This section can be skipped if the reader is interested in the \WP of $\Ext(A,G)$, only. 
We consider an  infinite group $G$ and we assume that $A$ is \nonarch, i.e., $A$ has rank at least 2.
We show that $\Ext(A,G)$ is never  torsion free. More precisely,
$\Ext(A,G)$ has always elements of order 2. Actually, often these elements generate 
$\Ext(A,G)$, see \refprop{gunnar}. Torsion elements which are not conjugated 
to torsion elements in $G$ can be represented as infinite  fixed points 
of the involution, i.e.,  by infinite closed words $x$ satisfying  $x = \ov x$, see 
\refprop{torsion}. In particular, all "new" torsion elements have order 2. 

According to Lemma~\ref{lem:totallang} there exists a (non-closed) partial 
word $p: \N \to \GG$, which is \Gred. This defines 
a closed word $[p)(\;\ov{p}]$ for each length $(m,1)$.
More formally, for $m \in \Z$ define 
\[
\begin{array}{rclll}
 w_m:[(0,0),(m,1)]&\to &\GG\\
 (n,0) &\mapsto& p(n) &\mbox{for }  n \geq 0 \\
 (n,1) &\mapsto& \ov {p(m-n)} &\mbox{for } n \leq m
 \end{array}
 \]
We have $\ov{w_m} = w_m$ and hence $w_m^2 = 1$ in $\Ext(A,G)$. 
By Theorem~\ref{thm:confl}  the element $w_m$ is not trivial, 
hence 
$w_m$ has order 2. 

In order to make the reasoning more transparent, 
assume that $G= F(\SS)$ is free. Then 
for $a\in \SS$ we may consider closed words  $w_m= [aaa\cdots\;)(\;\cdots \ov{a}\ov{a}\ov{a}] \in W(A,F(\SS)$. These words are pairwise different and
freely reduced. By \refcor{freeredS} reading  $w_m \in \Ext((A,F(\SS))$, these elements are still non-trivial,
pairwise different, and of order  2.

We have seen that $\Ext(A,G)$ contains infinitely many elements of order 2. Actually, frequently these elements generate $\Ext(A,G)$.

\begin{proposition}\label{gunnar}
Let $G= F(\SS)$
and $\abs \SS \geq 2$. Assume that $\Omega$ is a limit ordinal, 
that is for each $d \in \Omega$, we have $d +1  \in \Omega$, too. 
Then $\Ext(A,G)$ is generated by elements of order 2. 
\end{proposition}

\begin{proof}
 Let $x$ be cyclically  reduced with $\deg(x) = d$.  
(If $x$ is freely reduced, but $xx$ is not, then we can choose 
some $a \in \SS$ such that $xa$ is cyclically reduced since $\abs \SS \geq 2$.) 

We are going
to define a freely reduced word $x_\infty$ of length $t_{d+1}$ as follows. 
For $1 \leq \alp < t_{d+1}$ we let 
$x_\infty(\alp) = x^k(\alp)$, where $k\in \N$ is large enough that 
$\abs{x^k} \geq \alp$. 
Moreover, we let $x_\infty(t_{d+1}- \alp+1) = \ov{x_\infty(\alp)}$. 

Clearly, $x_\infty$ is freely reduced and $\ov{x_\infty} = x_\infty$, hence
$x_\infty$ is of order 2. Moreover, by construction, $x x_\infty = x_\infty \ov{x}$.
Hence, $x x_\infty$ has order 2, and $x = (x x_\infty)x_\infty$
is the product of two elements of order 2. Since $a_\infty$
is defined for $a \in \SS$, we see that all freely reduced words are 
a product of at most 4 elements of order 2. Now, freely reduced words generate
$\Ext(A, F(\SS))$, therefore elements of order 2 generate this group. 
\end{proof}


Clearly, as $G\sse\Ext(A,G)$, all torsion elements of $G$ appear in 
$\Ext(A,G)$ again, so we can conjugate them and have many more
torsion elements.

\begin{proposition}\label{torsion}
Let $x \in \Ext(A,G)$ be a torsion element
which is  not conjugated to any element in $G$. Then  
 there is a reduction $x \RAS*S y$ such that 
$y =\ov y$. In particular, we have  $x^2 = 1 \in \Ext(A,G)$. 
\end{proposition}

\begin{proof}
 Choose $x \RAS*S y$ such that $d\in \Omega$ is minimal and $\abs y = n_dt_d +\ell$ 
 with $\deg(\ell) < d$. Moreover, among these $y$ let the 
 leading coefficient $n_d \in \N$ be  minimal, too. Note that
 $y$ cannot contain any factor $u v \ov u$ where $\deg(v)<\deg(u) = d$ and 
 $v\RAS*S 1$. Since $x$ has torsion, we may assume  $x^k= 1 \in \Ext(A,G)$ for some
 $k >1$. Hence $y^k\RAS*S 1$ due to  confluence of $S$. Now, 
 $\deg(y^k) = d$, hence $y^k\RAS*S 1$ implies that $y^k$ has a factor
 $u v \ov u$ where $\deg(v)<\deg(u) = d$ and 
 $v\RAS*S 1$. Making $v$ larger and $u$ (and $\ov u$) smaller, we can factorize $ v = v_1v_2$ such that
 $uv_1$ is a suffix of $y$ and $v_2\ov u$ is a prefix of $y$. 
 Moreover, for some closed word $z$ of degree $d$ we have 
 $uv_1 \RAS*S z \LAS*S u\ov{v_2}$. Hence we can assume that $z$ is a suffix of $y$ and $\ov z$ is a prefix of $y$. If $z$ and $\ov z$ overlap in $y$ (that
 is $\abs y < 2 \abs z $), then we have $y =\ov y$. 
 Otherwise we  write $y = \ov z y' z$ and
 we replace $x$ by $y'$ and we use induction. 
 \end{proof}
%

\section{Group extensions over $A=\Z\bracket{t}$}\label{sec:ztextensions}
For the remainder of the present paper we assume that $A=\Z\bracket{t}$.
This  means $A$ is the additive group of the polynomial ring over $\Z$ in one variable $t$. The reason for the choice of $A$ is that we wish the 
subgroup $A^{\deg<d}$ to be finitely generated for each degree $d\in \Omega$ where:  
 $$ A^{\deg<d} = \set{\beta \in A}{\deg(\beta) <d}.$$

 This assumption is clear for $A=\Z\bracket{t}$, 
 because each such subgroup is isomorphic 
 to $\Z^d$ with $d \in \N$. 
 Moreover, every finitely generated subgroup $H$ 
 of  $\Ext(A,G)$ sits inside some  $\Ext(\Z^d,G)$.
 
 We shall  use  
 the following well-known fact:

\begin{lemma}\label{ascendings}
Let $k\geq 0$ and $$A_0 \subseteq A_1 \subseteq A_2 \subseteq A_3 \cdots$$
 be an infinite ascending chain of subgroups in  $\Z^k$. Then this
 chain becomes stationary, i.e., there is some $m$ such that 
 $A_m = A_n$ for all 
all $n \geq m$. 
\end{lemma}

\subsection{Proper periods}\label{sec:proper}

Let $w \in W(A,\GG)$ be a word of length $\alpha \in A$, given as 
a mapping 
$w:[1,\alpha] \to \GG$.
An element $\pi \in A$ is called a \emph{period} of $w$,  
if 
for all $\beta \in A$ such that 
$1 \leq \beta,\, \beta + \pi \leq \alpha$ we have 
$$w(\beta)= w(\beta + \pi).$$

A period $\pi$ is called a \emph{proper period} of $w$,  
if $\deg(\pi) < \deg (w)$. In the following we are interested in proper periods, only. 
We have the following basic lemma. 

\begin{lemma}\label{lem:proper}
Let $w \in W(A,\GG)$ of degree $\deg(w) = d$ with $0\leq d$, then the set $\Pi(w)$ of proper 
periods forms a subgroup of $A^{\deg<d}$. 
\end{lemma}

\begin{proof}
We have  $0 \in \Pi(w)$. If 
$\pi \in \Pi(w)$, then $-\pi \in \Pi(w)$, too. Let 
$\pi', \pi \in \Pi(w)$ with $0 \leq \pi' \leq \pi$.
Clearly, $\pi+\pi'$ is a proper period, too.
It remains to show that $\pi-\pi'$ is a proper period. 
To see this, let $\beta \in A$ such that 
$0 \leq \beta,\, \beta + \pi - \pi'\leq \abs w$. For 
$\beta + \pi \leq \abs w$ the element 
$\pi-\pi'$ is a proper period, because
then $w(\beta)= w(\beta + \pi) = w(\beta + \pi - \pi').$
Hence we may assume that $\beta + \pi > \abs w$. 
But $\deg(\pi) < \deg (w) $, hence $\deg(\bet) = \deg (w)$ and therefore   $0 \leq \beta - \pi'$. Thus,
$w(\beta)= w(\beta - \pi') = w(\beta + \pi - \pi').$
\end{proof}

Together with \reflem{ascendings} the lemma above leads
us to the following observation: 

\begin{proposition}\label{prop:proper}
Let $w_0, w_1, w_2, w_ 3, \ldots$
 be an infinite sequence of elements of $W(A,\GG)$
 such that $w_{i+1}$ is always a non-empty factor of $w_{i}$. Let 
 $$\Pi_0,\Pi_1, \Pi _2, \Pi_3, \ldots$$
 be the corresponding sequence of proper periods  in  $A$. 
 Then this
 sequence  of groups becomes stationary, i.e., there is some $m$ such that 
 $\Pi_m = \Pi_n$ for all 
all $n \geq m$. 
\end{proposition}

\begin{proof}The sequence of degrees is descending and becomes 
stationary. Hence we may assume that in fact 
$$0 \leq \deg(w_0) =\deg( w_1) = \deg (w_2) = \deg (w_ 3) = \cdots.$$
As a consequence  $$\Pi_0 \subseteq \Pi_1 \subseteq \Pi_2 \subseteq \Pi_3 \cdots$$
 is an ascending chain of subgroups in some $\Z^k$
 which becomes therefore stationary.
 \end{proof}
\section{Deciding the \WP in $\Ext(A,G)$}\label{sec:wp}
Recall that for a finitely generated group the decidability of the \WP
does not depend on the presentation: It is a property of the group. 
In the following we restrict ourselves to the case  that $\GG$ is finite
(in particular, 
$G$ is finitely generated).
The main difficulty for deciding the \WP 
in $\Ext(A,G)$ is due to periodicity.

\subsection{Computing reduced degrees}
Let $S$ be the  system defined in Equation~\ref{eq:hugo} which 
is  confluent by  \refthm{thm:confl}.
If we have $x\RAS * S y$ then we have $\deg(x) \geq \deg(y)$. 
Thus, we can define the \emph{reduced degree} by
$$ \rdeg x = \min \set {\deg(y)}{x\RAS * S y}.$$
Note that $ \rdeg x$ is well-defined for group elements
$x \in \Ext(A,G)$ due the confluence of $S$.

\begin{lemma}\label{karl}
Let $u\in R(A,G)$ be a non-empty \Gred word. Then we have 
$0 \leq \deg(u) = \rdeg u$. 
\end{lemma}

\begin{proof}
 This is a direct consequence of \reflem{lem:otto}. 
\end{proof}

Clearly, since $G$ is a subgroup of $\Ext(A,G)$, the 
\WP of $G$ must be decidable, otherwise we cannot 
hope to decide the \WP for finitely generated subgroups of  $\Ext(A,G)$. 

Our goal is to solve the \WP in $\Ext(A,G)$
via the following strategy. We 
 compute on input $w \in W(A,\GG)$
some $w' \in W(A,\GG)$ such that both $w\DAS * S w'$ and
$\deg(w') = \rdeg w$.
If $\deg(w') > 0$, then $w\neq 1$ in $\Ext(A,G)$.
Otherwise $w'$ is a finite word over $\GG$ and
we can use the algorithm for $G$ which decides whether 
or not $w'= 1$ in $G\subseteq \Ext(A,G)$.

In order to achieve this goal we need a slightly stronger
condition on $G$. We need that the non-uniform cyclic membership problem in $G$ is decidable. This means that for each $v\in \GG^*$ there
is an algorithm $\cA(v)$ which solves the problem "$u \in \gen{v}$?".
Thus, $\cA(v)$ decides on input $u\in \GG^*$
whether or not $u$ (as an element of $G$) is in the subgroup 
of $G$ which is generated by $v$. This requirement on $G$ is indeed a necessary condition:

\begin{theorem}\label{neccond}
Assume that the \WP 
is decidable for each finitely generated subgroup 
of $\Ext(A,G)$. Then  for each $v\in \GG^*$ there
exists an algorithm which decides on input $u\in \GG^*$
whether or not $u$ (as an element of $G$) is in the subgroup 
of $G$ which is generated by $v$. 
\end{theorem}

\begin{proof} 
Let $v\in \GG^*$ be a finite word. If $v$ is empty we are done because 
"$u\in \gen{1}$?" is nothing but the \WP for $G$ (which is a finitely generated subgroup 
of $\Ext(A,G)$). Hence we may assume that $v$ is non-empty and moreover, 
$v\neq 1$ in $G$. 
If $v$ is a torsion element, then the question 
whether or not $u$  is in the subgroup 
 generated by $v$ can be reduced to the \WP.
Hence may assume that $v^k \neq 1$ for all $k\neq 0$. 
We perform an induction on the length of $v$ which allows to view $v$ as a finite \Gred word.

We can solve the problem 
"$u\in \gen{v}$?" for all inputs $u$ as soon as we can solve the problem 
"$u\in \gen{pv^k\ov p}$?" for some $p$ and $k \neq 0$ for all inputs $u$.
Indeed, fix $p$ and $k$. Then, $u \in \gen{v}$ \IFF{}
$puv^i\ov p \in \gen{pv^k\ov p}$ for some $ 0 \leq i <\abs k$. 
Clearly, $puv^i\ov p \in \gen{pv^k\ov p}$ implies $u \in \gen{v}$.
For the other direction let $u = v^m$. We can write 
$m = \ell k -i$ with $\ell \in \Z$ and $ 0 \leq i <\abs k$.
It follows $puv^i\ov p \in \gen{pv^k\ov p}$.
Thus, the problem 
"$u\in \gen{v}$?" is reduced to the 
problem:
$$\mbox{"} \exists i: \, 0 \leq i <\abs k \quad \& \quad puv^i\ov p \in \gen{pv^k\ov p}\mbox{?"}$$
Therefore,  by induction on $\abs v$ we may assume 
that no proper factor $w$ of the word $v$ is equal to any  
$pv^k\ov p$ in $G$. (We only need the 
existence of an algorithm. There is no need to construct 
the algorithm on input $v$.) 

Next, we claim that every power $v^m$ is \Gred.
Assume the contrary, then there are words $p,q,r,s$ and $k \in \N$
such that 
$v= pq = rs$ and $q \neq v \neq r$ as words, but $qv^kr = 1$ in $G$. 
Note that neither $r$ nor $q$ can be the empty word by the induction 
hypothesis. Moreover, $p \neq r$ because $v^{k+1} \neq 1$ in $G$. 
If $\abs{p} < \abs{r}$, then we can write $r= pw$ where $w$ is a proper factor of $v$,
and we obtain 
$$1= qv^kpw = \ov p p qv^kpw = \ov p v^{k+1} p w.$$
This is impossible
since no proper factor of $v$ is of the form $ p v^{-k-1} \ov p$ in $G$.

If $\abs{p} > \abs{r}$, then $p= rw$ for some proper
factor $w$ of $v$. We obtain $qv^kp = w$ in $G$.
Again this is impossible, because it would imply 
$qv^kpq \ov q = q v^{k+1} \ov q = w$ in $G$.

Thus,  $V= [vvv\cdots\;)(\;\cdots vvv]$ is a \Gred word 
of degree 1 in $\Ext(A,G)$. Next, we may assume that $v$ is a primitive word, this means $v$  is no proper power of 
any other word. It follows that $v$ does not appear properly inside
$vv$ as a factor. 

We claim that now,  $u \in \gen{v}$ \IFF{} $uV=Vu$ in  $\Ext(A,G)$.
Clearly, if 
$u \in \gen{v}$ then $uV=Vu$ in  $\Ext(A,G)$. For the other 
direction let  $uV=Vu$ in  $\Ext(A,G)$. Then by applying 
finitely many times defining relations for $G$ we must be able to
transform the one-sided partial infinite word $u\, vvv\cdots\;$ into 
$vvv\cdots\;$. Thus for some word $w\in \GG^*$, a factorization 
$v= pq$,  and $k,\ell\in \N$ we obtain
$uv^k = v^\ell p$ in $G$ such that the infinite words 
$w vvv \cdots\;$ and   $wqvvv\cdots\;$ are equal. But $v$ is primitive
and hence $p\in \oneset{1,v}$. Thus,  $u \in \gen{v}$. 
\end{proof}

\begin{theorem}\label{thm:decwp}
 Let $G$ be a group such that for each $v\in \GG^*$ there
is an algorithm which decides on input $u\in \GG^*$
whether or not $u\in G $ is in the subgroup 
of $G$ generated by $v$. 

Then for each finite subset $\Delta \subseteq  W(A,\GG)$ 
of \Gred words (i.e., $\Delta \subseteq  R(A,G)$) there 
is an algorithm which computes on input $w \in \Delta^*$ its reduced degree and
some $w' \in W(A,\GG)$ such that both $w\DAS * S w'$ and
$\deg(w') = \rdeg w$.
\end{theorem}

\begin{proof}
The proof is split into two parts.
The first part is a preprocessing 
on the finite set $\Delta$. In the second part we 
present the algorithm for the set $\Delta$ after the preprocessing.
\bigskip

{\bf PART I: Preprocessing}

The preprocessing concerns $\Delta$ and not the actual algorithm.
Therefore it is not an issue  that the steps 
in the preprocessing are effective.
It is clear that we may replace  $\Delta$ by any other finite set
$\widehat\Delta$ 
such that $\Delta \subseteq \widehat\Delta^*$. 
This is what we do. 
We  apply the following transformation rules in any order as long as possible, and we stop if no rule changes $\Delta$  anymore. 
The result is $\widehat\Delta$ which is, as we will see, still a set of \Gred words. (This will follow from the fact that 
every  factor of a \Gred word is \Gred).

\begin{enumerate}[1.)]
\item Replace $ \Delta$ by $( \Delta \cup \GG) \setminus \oneset 1$.
(Recall that $\GG$ is finite in this section.)
\item If we have $g \in \Delta$, but  $\ov g \not\in \Delta$, then
insert $\ov g$ to $ \Delta$.
\item If we have $g \in \Delta$ with $g= fh $ in $W(A,\GG)$
and $\deg(g) = \deg(f) = \deg (h)$, then remove $g$ and $\ov g$ 
{}from  $\Delta$ and 
insert $f$ and $h$ to $ \Delta$.
\end{enumerate}

After these steps every element  in $\Delta$ has
its inverse in $\Delta$ and for some $d \in \N$ it has a length of the form 
$t^d +\ell$ with $\deg(\ell) < d$. Thus, the leading coefficient
is always 1. In particular, all generators of finite length are
letters of $\GG = \Sigma \cup \ov \Sigma$. 
The next rules  are more involved. We first define an equivalence 
relation on $W(A,\GG)$. We let $g \sim h$ if
for some $x,y,z,t$,  and $u$ in $W(A,\GG)$ with $\deg(xyzt) < \deg(u)$ we have 
$$ g = xuy \quad \mbox { and } \quad h = zut.$$
Note that the condition implies $\deg (g) = \deg(u) = \deg(h)$.
The effect of the next rule is that for each equivalence class there
is at most one group generator in $\Delta$.

\begin{enumerate}
\item[4.)] If we have $g, h \in \Delta$ with $ g \not\in \oneset {h, \ov h}$, 
but $ g = xuy$ and $ h = zut$ for 
some $x,y,z,t$,  and $u$ with $\deg(xyzt) < \deg(u)$, then 
remove $g, h, \ov g, \ov h$ {}from  $\Delta$ and 
insert $x,y,z,t$ (those which  are non-empty) and $u$ to $ \Delta$.
\end{enumerate}

\begin{enumerate}
\item[5.)] If we have $g \in \Delta$ with $ g \neq \ov g$, 
but $ g = xuy = z\ov u t$ for 
some $x,y,z,t$,  and $u$ with $\deg(xyzt) < \deg(u)$, then 
write $ u= pq$ with $\deg(p) < \deg(q)= \deg(u)$
and $q = \ov q$. 
Remove $g$ and $\ov g$ {}from  $\Delta$ and 
insert $x,y,z,t, p, \ov p$ (those which  are non-empty) and $q$ to $ \Delta$.
(Note that $g \sim q$.)
\end{enumerate}

The next  rules deal with periods. 

\begin{enumerate}
\item[6.)] If we have $g\in \Delta$ 
and  $ g = xuy$ for 
some $x,y$,  and $u$ with $\deg(xy) < \deg(u)$
such that $u$ has a proper period which is not a period of $g$, then 
remove $g, \ov g $ {}from  $\Delta$ and 
insert $x,y$ (those which  are non-empty) and $u$ to $ \Delta$.
\end{enumerate}

The following final  rule below makes $\Delta$ larger again, and the rule  adds additional
information to each generator. For each $g \in \Delta$ let 
$\Pi(g)\subseteq A$ the group of proper periods. Let
$B(g)$ be a  set of generators of $\Pi(g)$. 
We may assume that for each possible degree $d$ there is at most one
element $\beta \in B(g)$ of degree $d$. Moreover, we may assume 
$0 \leq \beta $ and for each $g$ the set $B(g)$ is fixed. In particular, for $\pi \in \Pi(g)$
with $\deg(\pi) = d \geq 0$ there is exactly one  $\beta \in B(g)$
such that $\deg(\beta) = d$ and $\pi = m \beta + \ell$ 
for some unique  $m\in \Z$ and $\ell \in \Pi(g)$ with 
 $\deg(\ell) < d$. For each $\beta \in B(g)$
let $r(\beta)$ be the prefix and $s(\beta)$ be the suffix of length $\beta$ 
of $g$. (In particular, $r(\beta)\,g= g\,s(\beta)$ in $W(A,\GG)$.) Note that the number of $r(\beta)$, $s(\beta)$ 
is bounded by $2\deg(g)$. 

\begin{enumerate}
\item[7.)] If we have $g\in \Delta$, 
then let  $B(g)$ be a set of generators 
for the set of proper periods $\Pi(g)$ as above. If necessary, enlarge
$\Delta$ by finitely many elements of degree less than $\deg(g)$ 
(and which are factors of elements of $\Delta$) such 
that  $r(\beta)$, $s(\beta) \in \Delta^*$ for all $\beta \in B(g)$.
\end{enumerate}

Note that the rules 1.) to 7.) can be applied only a finite number of times. The formal proof relies on K\"onig's Lemma and \refprop{prop:proper}.

\begin{remark}\label{rem:mainshort}
Note that the preprocessing has been  done in such a way that every
element in $\widehat\Delta$ is either a letter or a factor of an element
in the original  set $\Delta$. In particular, if $\Delta$ contains 
local geodesics only, then $\widehat\Delta$ has the same property. This fact is 
used for \refcor{cor:mainshort}. 
\end{remark}

\bigskip
{\bf PART II: An algorithm to compute the reduced degree}

We 
may assume that $\Delta$ has passed the preprocessing, i.e.,
$\Delta = \widehat \Delta$ and
no rule above changes $\Delta$ anymore. 
 The input $w$ (to the 
 algorithm we are looking for) is given as a word $g_1 \cdots g_n$ with $g_i \in \Delta$.
 Let $$d = \max\set{\deg(g_i)}{1 \leq i \leq n }.$$ 
 We may assume that $d>0$. Either 
 $\deg(w) = \rdeg w$ (and we are done) or $\deg(w) >\rdeg w$
 and $w \in W(A,\GG)$ contains a factor $uv\ov{u}$ such that 
 the following conditions hold: 
 
\begin{enumerate}[1.)]
\item The word $u $ is \Gred and has 
length $\abs u = t^d + \ell$ with $\deg (\ell ) < d$,
\item $\deg (v) < d$,
\item $ v \RAS *S 1$.
\end{enumerate}

We may assume that the factor $uv\ov{u}$ starts in some $g_i$ 
and ends in some $g_j$ with $i<j$,
because the leading coefficient of each length $\abs {g_i}\in \Z \bracket t$ is $1$. Moreover, by making $u$ smaller and thereby 
$v$ larger,
we may in fact assume that $u$ is a factor of  $g_i$ 
and $\ov{u} $ is a factor of $g_j$. Thus, $\deg(g_i) 
= \deg(u)= \deg(g_j)= d$ and
we can write $g_i = xuy$ and $g_j = z \ov u t$. By preprocessing on $\Delta$ (Rule 4),
we must have $g_i \in \oneset{g_j, \ov{g_j}}$.
Assume $g_i =g_j$, then
we have $g_i = xuy =  z \ov u t$ and, by preprocessing on $\Delta$ (Rule 5), we may conclude $g_i =\ov {g_i}$. Thus in any case we know $g_i =\ov {g_j}$.

Thus, henceforth we can assume that for some $1 \leq i<j \leq n$ 
 we 
have in addition to the above:

\begin{enumerate}
\item [4.)]$g_i = xuy$,
\item[5.)] $ v = y g_{i+1} \cdots g_{j-1} z$,
\item[6.)] $g_j = z \ov u t = \ov{g_i}$.
\end{enumerate}
Since $g_j = z \ov u t = \ov{g_i}$ we have $xuy=  \ov t u \ov z$, 
and by symmetry (in $i$ and $j$) we many assume:
\begin{enumerate}
\item [7.)] $\abs y \geq \abs z$.
\end{enumerate}
This implies $y  = q \ov z$ for some $q \in W(A,\GG)$
with $\deg(q) < d$ and $uq = q'u$ for $\ov t = xq'$. 

Therefore $\abs q$ is a proper period of $u$, and hence, 
 by preprocessing on $\Delta$ (Rule~6), 
 we see that  $\abs q$ is a proper period of $g_i$.
 Thus there are $p',p \in \Delta^*$ with $\abs {p'} = \abs p = \abs q$
 such that $p'g_i = g_i p$. But $ \ov z$ and $y$ are suffixes of $g_i$,
 hence 
 $$y = \ov z p.$$
 Therefore:
 \begin{enumerate}[8.)]
\item $ p g_{i+1} \cdots g_{j-1} \RAS *S 1$, where $p$ is a suffix of $g_i$ and 
$\abs p$ is a proper period of $g_i$.
\end{enumerate}

We know $ \deg(g_{i+1} \cdots g_{j-1}) < d$. Hence by induction 
on $d$ we can compute $h \in \Delta^*$ such that both 
$g_{i+1} \cdots g_{j-1} \DAS * S h$ and
$\deg(h) = \rdeg {g_{i+1} \cdots g_{j-1}} $. This implies
$\deg(h) = \rdeg p $, too. But $p$ is a factor of a \Gred word, 
hence actually $\deg(h) = \deg(p) $ by \reflem{karl}.

We distinguish two cases. Assume first that $\deg(h) \leq 0$.
Then $h, p  \in \GG^*$ are finite words. If $h=1$ in $G$, then 
we can replace the input word $w$ by 
$$ g_1 \cdots g_{i-1} g_{j+1} \cdots g_n$$
since $  g_i g_{i+1} \cdots g_{j-1}\ov {g_i}\RAS *S 1$, and we are done 
by induction on $n$. 

If $ h \in \GG^*$ is a finite word, but $h\neq 1$ in $G$, then
$p = h^{-1} \neq 1$ in $G$, too. Consider   the smallest element
$\rho \in B(g_i)$ and let 
$r\in \GG^*$ be  the suffix of $g_i$  with $\abs r = \rho$. It follows that $p$ is a positive power of $r$ because $\abs p$ is a  period of $g_i$. 
This means that $h$ is in the subgroup of $G$ generated by $r$. 
For this test we have an algorithm 
by our hypothesis on $G$. According to our assumptions 
the answer of the algorithm is \emph{yes}: $h$ is in the subgroup generated by $r$.
This allows to find $m\in \Z$ with $\ov h = r^m$ in the group $G$.
We find some finite word $s$ of length $\abs s = \abs{r^m} $ 
such that  $sg_i= g_i r^m$;  and we can replace the 
input word $w$ by $g_1 \cdots g_{i-1} \ov{s}  g_{j+1} \cdots g_n$, 
because we have: 
\begin{align*}
g_1 \cdots  g_{i}  \cdots g_j \cdots g_n
 & \DAS * S g_1 \cdots \ov{s} s g_{i}  \cdots g_j \cdots g_n\\
 & \DAS * S 
g_1 \cdots \ov{s} g_{i} r^m  h \ov {g_i} g_{j+1} \cdots g_n\\ 
&\RAS * S g_1 \cdots \ov{s} g_{i}  \ov {g_i} g_{j+1} \cdots g_n\\
& \RAS {} S g_1 \cdots g_{i-1} \ov{s}  g_{j+1} \cdots g_n. 
\end{align*}
We are done by induction on the number of generators of degree $d$. 

The final case is  $\deg(h) > 0$. We write $\abs h = m't^e + \ell$
with $\deg (\ell) < e = \deg(h). $
 According to our preprocessing on $\Delta$ (Rule 7) 
there are  words $r,s\in \Delta^*$ such that $\deg (r) = \deg (p)$,
$r$ is a suffix 
of $g_i$
with $sg_i= g_i r$. For some $m$ with $m \leq m' $ we must have 
$\rdeg {r^m h} < e$. By induction we can compute some word $f$ 
with $\deg(f) = \rdeg {r^m h}$ and $f \DAS* S {r^m h}$.
Like above we can replace the 
input word $w$ by 
$$g_1 \cdots g_{i-1} \ov{s}^m g_i f  g_{j} \cdots g_n,$$ 
because  we have: 
\begin{align*}
 g_1 \cdots \ov{s}^m s^m g_{i}  \cdots g_j \cdots g_n
 & \DAS * S
g_1 \cdots \ov{s}^m g_{i} r^m  h \ov {g_i} g_{j+1} \cdots g_n\\ 
&\DAS * S  g_1 \cdots \ov{s}^m g_{i} f \ov {g_i}g_{j+1} \cdots g_n.
\end{align*}
We are done by induction on the degree $e$ which is the 
reduced degree of the factor $r^mg_{i+1} \cdots g_{j-1}$. We can apply this 
induction since  $g_i r^mg_{i+1}\cdots  g_{j-1}g_j$
now has  a factor $uv\ov{u}$ such that 
 the following conditions hold: 
 \begin{enumerate}[1.)]
\item The word $u $ is \Gred and $\deg (u) = d> 0$,
\item $\deg (v) < e$,
\item $ v \RAS *S 1$.
\end{enumerate}
\end{proof}

By  Theorems~\ref{neccond} and~\ref{thm:decwp} 
we obtain the following corollary which gives the 
precise answer in terms of the group $G$ whether or not the 
\WP in finitely generated subgroups of $\Ext(A,G)$
is decidable.

\begin{corollary}\label{cor:main}
Let $G$ be finitely generated by $\GG$ and  $A=\Z\bracket{t}$. 
Then the following assertions are equivalent:

\begin{enumerate}[i.)]
\item  For each $v\in \GG^*$ there
is an algorithm which decides on input $u\in \GG^*$
the Cyclic Membership Problem "$u \in \gen{v}$?"

\item For each finite subset $\Delta \subseteq  W(A,\GG)$ there 
is an algorithm which decides  on input $w \in \Delta^*$
whether or not $w = 1 $ in the group $\Ext(A,G)$.
\end{enumerate}

\end{corollary}

Recall that (according to 
Definition~\ref{def:short})
a local geodesic denotes word 
without any finite factor $f$ such that $f=g$ in $G$ but $\abs{g}<\abs{f}$.
Inspecting the proof above we find the following variant 
of \refcor{cor:main}. 

\begin{corollary}\label{cor:mainshort}
Let $G$ be finitely generated by $\GG$ and  $A=\Z\bracket{t}$. 
Then the following assertions are equivalent:

\begin{enumerate}[i.)]
\item  The group $G$ has a decidable \WP. 
\item For each finite subset $\Delta \subseteq  W(A,\GG)$ 
of  local geodesics there 
is an algorithm which decides  on input $w \in \Delta^*$
whether or not $w = 1 $ in the group $\Ext(A,G)$.
\end{enumerate}
\end{corollary}

\begin{remark}\label{rem:sapir}
Clearly, Condition i.) in Corollary~\ref{cor:main} implies
Condition i.) in Corollary~\ref{cor:mainshort}, but the converse fails.
There is a finitely presented 
group $G$ with a decidable \WP, but one can construct  a specific
word $v$ such that the Cyclic Membership Problem "$u \in \gen{v}$?" is undecidable, see \cite{olsap00,olsap01}.
\end{remark}

\begin{remark}\label{cyclicmemb}
Let $G$  be a finitely generated group. Of course, if $G$ has a decidable 
\emph{\gwp,} i.e., the 
Membership Problem w.r.t.{} finitely generated subgroups is decidable, then 
the Cyclic Membership Problem "$u \in \gen{v}$?"
is decidable, too. Examples of groups $G$ where the \gwp is decidable
include metabelian, nilpotent or, more general, abelian by nilpotent groups, 
see \cite{romanovskii}. However, there are also large classes of 
groups, where the Membership Problem is undecidable, but the Cyclic Membership Problem is  easy. For example, the Cyclic Membership Problem is 
decidable in linear time in a direct product of free groups, but
as soon as $G$ contains a direct product of free groups of rank 2, the \gwp 
becomes undecidable
by \cite{Mihailova58}. For hyperbolic groups a construction of Rips shows that the \gwp is 
undecidable (\cite{rips82}), but the Cyclic Membership Problem "$u \in \gen{v}$?"
is decidable by \cite{Lys90}. 

Decidability of the the Cyclic Membership Problem is also preserved
e.g. by effective HNN extensions. This means, if $H$ is an HNN-extension of
$G$ by a stable letter $t$ such that we can effectively compute Britton reduced forms, then one can reduce
the Cyclic Membership Problem "$u \in \gen{v}$?" in $H$ to the same problem in 
$G$ as follows. On input $u,v$ we compute first the  Britton reduced form of $v$.
This tells us whether $v\in G$. If so, we are done by checking first that $u\in G$
and then by using the algorithm for $G$. So, let $v\in H\sm G$.
Via conjugation we may assume that $v^k$ remains Britton reduced for all $ k \in \Z$.
Now, if $u$ is Britton reduced, too, then it is enough to check $u = v^k$
for that $k$ where the $t$-sequence of $u$ coincides with the one of $v^k$.
There is at most one such $k$. Thus we can use the algorithm to decide the
\WP in $H$ which exists because we can effectively compute Britton reduced forms. 

As every one-relator group $G$ sits inside an 
effective HNN extension of another one-relator group with a shorter relator
\cite{LS}, 
we see that the Cyclic Membership Problem is decidable in one-relator groups, too. 
The property is also preserved by effective amalgamated products for a similar 
reason as for HNN extensions. 
\end{remark}
%


\section{Realization of some HNN-extensions}\label{sec:HNN}
The purpose of this section is to show that the group $\Ext(A,G)$
contains some important HNN-extensions of $G$ which therefore can be studied 
within the framework of infinite words. Moreover, we show that $\Ext(A,G)$
realizes more HNN-extensions than it is possible in the 
approach of  \cite{MRS05}. The reason is that \cite{MRS05}
is  working with cyclically reduced decompositions, only. 
We begin with this concept and we show first how it embeds in our setting.
 
\subsection{Cyclically decompositions for freely reduced words}\label{sec:cdr}
In \cite{MRS05} a partial multiplication on freely reduced words 
and a partial monoid $\cdr(A,\SS)$ has been defined 
for a free group $F(\SS)$:  
Let $x,y \in R(A,F(\SS))$ be freely reduced words. 
The partial multiplication $x * y$ is defined \IFF $ x = p q$, $y =\ov q r$, and
 $pr$ is freely reduced. In this case  $x * y = pr$. 

As a set  $\cdr(A,\SS)$ consists of those freely reduced words $x$, which admit a \emph{cyclically reduced decomposition}
 $x = cu\ov c$ where $u$ is cyclically reduced. 
 If the decomposition exists, it  is unique. Note that 
 $c = [aaa\cdots\;)(\;\cdots \ov{a}\ov{a}\ov{a}]$ is freely reduced, but it is not
 in $\cdr(A,\SS)$. On the other hand, for $a \neq b \in \SS$ we have 
 $$x = [aaa\cdots\;)(\;\cdots \ov a \ov a \ov a b aaa \cdots\;)(\;\cdots \ov{a}\ov{a}\ov{a}] \in \cdr(A,\SS)$$ since $x = c b \ov c$.

 In terms of the group $\Ext(A,F(\SS))$ we can rephrase this as follows. 
  The set $\cdr(A,\SS)$ embeds into $\Ext(A,F(\SS))$ because all elements are freely reduced and hence
  irreducible by the confluent system $S$. Now,
  $\cdr(A,\SS)$ (being a subset of a group) becomes a partial monoid by restricting
  the definition of  $x * y$ to the  case where $x * y = xy$ is defined and
   $xy \in\cdr(A,\SS)$. If  $xy \notin\cdr(A,\SS)$, then the result $x * y$
   remains undefined. 
   
   Now, assume $x,y$, and 
 $xy \in \cdr(A,\SS)$. Then there exists a freely reduced word  $z = cw\ov c$ where $w$ is cyclically reduced such that $xy \RAS*S z$. The reduction provides us with 
 a factorization such that $x = p q$, $y =\ov q r$, and
 $pr$ is freely reduced. Thus, $x * y$ is defined. In this way the partial monoid 
 $\cdr(A,\SS)$ embeds naturally into the group $\Ext(A,F(\SS))$. 
 
 Let $a,b \in \SS$ with $a \neq b$. It is known that the HNN-extension of $F(\SS)$ by 
 $sbs^{-1}= a$ with stable letter $s$ embeds into $\cdr(A,\SS)$
 with $s =  [aaa\cdots\;)(\;\cdots bbb]$. To see this, observe that
 this HNN-extension can be written as a semi-direct product $F(a,b)\rtimes \Z$.
 This allows to write elements in normal form as a word $x= w\cdot s^k$
 where $w$ is a freely reduced word over $\SS^{\pm}$ and $k \in \Z$. 
 A direct inspection shows that $x$ is in $\cdr(A,\SS)$ and it is trivial
 $\Ext(A,F(\SS))$ \IFF it is trivial in $F(a,b)\rtimes \Z$.

 However, 
  the HNN-extension $H$ of $G$ by 
 $sb^2s^{-1}= a^2$  does not embed into $\cdr(A,\SS)$ because the 
 commutation relation $\sim$ is not transitive, but it is known to be transitive in any
 finitely generated subgroup of $\cdr(A,\SS)$, \cite{bass91}. The 
 commutation relation  is not transitive in $H$, because 
 $a \sim a^2 = sb^2s^{-1} \sim sbs^{-1}$, but $a \not\sim sbs^{-1}$ in $H$. 
 
 The group  $\Ext(A,F(\SS))$ is however large enough to realize the HNN extension $H$, but we have to leave $\cdr(A,\SS)$:  
 Define 
 $$s=  [aaa\cdots\;)(\;\cdots ababab \cdots\;)(\;\cdots bbb].$$
 Then the canonical \homo $H\to \Ext(A,F(\SS))$ is an embedding. (See \refprop{HNN}.)
 Note that $sb^2s^{-1}= a^2$, but $sbs^{-1}\neq  a$ due to the middle line of
 $ab$'s which requires a shift by 2 in order to be matched. Clearly, 
 $s, b, \ov s \in \cdr(A,\SS)$ and $s' = s* b\in  \cdr(A,\SS)$ is defined. 
 But $s' * \ov s$ is not defined, and therefore $sbs^{-1} \in \Ext(A,F(\SS))\sm \cdr(A,\SS)$. (Note that $s\cdot b\cdot  s^{-1}$ is not a cyclically reduced decomposition, because $sb \ov s$ is not freely reduced and 
there is no freely reduced word $x$ such that $x = sb s^{-1} \in \Ext(A,F(\SS))$.)
The element $s'\ov s = s b \ov s$ can be depicted as follows: 
 $$sb\ov s =  [aaa\cdots\;)(\;\cdots ababab \cdots\;)(\;\cdots bbb]
 [\ov b \ov b \cdots\;)(\;\cdots bababa \cdots\;)(\;\cdots bb]b.$$

\begin{remark}\label{split}
Let $H$ be a subgroup inside the partial monoid $\cdr(A,\SS)$, then 
$H$ is torsion-free. Indeed $(c u \ov c)^2 = c u^2 \ov c$ and we can use 
\refprop{torsion}. Since $sbs^{-1}$ is torsion-free and $sbs^{-1} \in \Ext(A,F(\SS))\sm \cdr(A,\SS)$
for $s$ and $b$ as above, we see that the set of torsion elements is a 
proper subset of $\Ext(A,F(\SS))\sm \cdr(A,\SS)$, in general. 
\end{remark}

We conclude this subsection with a few more examples which allow similar
calculations as above. In these examples we use however stable letters 
which have no cyclically reduced decomposition.

\begin{example}\label{ex:semidirect}
Consider the
following non-abelian semi-direct products: 
$G_1= \Z\rtimes (\Z/ 2 \Z) $ (which is isomorphic to the free product 
$\Z/ 2 \Z * \Z/ 2 \Z$) and $G_2= \Z\rtimes \Z$. 
(which is isomorphic to the Baumslag-Solitar group $\BS(1,-1)$.)
The groups $G_1$ and $G_2$ can be embedded 
into $\Ext(\Z\times \Z, \Z)$.
Indeed, define  $s_1$ and $s_2$ as follows: 
\begin{align*}
s_1 &= [aaa\cdots\;)(\;\cdots \ov a \ov a \ov a],\\
s_2 &= [aaa\cdots\;)(\;\cdots aaaa \cdots\;)(\;\cdots \ov a \ov a \ov a].
\end{align*}
The element $s_1$ has order 2 and $s_2$ has infinite order 
in $\Ext(\Z\times \Z, \Z)$. 
Clearly $a s_i = s_i \ov a$, and it is easy to verify
that the subgroups generated by $a$ and $s_i$ are isomorphic to $G_i$ 
 for $i= 1,2$.

Let $\SS \geq 2$ and let $G_3$ be  the HNN-extension of $\Z$ 
with stable letter $s$ and defining relation 
$s^{-1}a^2 s = a^{-2}$.
The group $G_3$ is also the Baumslag-Solitar group $\BS(2,-2)$.
It embeds into $\Ext(\Z\times \Z, F(\SS))$ using $s_3$ as a stable letter, where
\begin{align*}
s_3 &=  [aaa\cdots\;)(\;\cdots ababab \cdots\;)(\;\cdots \ov a \ov a \ov a].
\end{align*}
Again, a direct verification that this group embeds
is not difficult. 
All three embeddings
 occur as  special cases of \refprop{HNN}. None of these 
groups can be embedded into the partial monoid  $\cdr(A,\SS)$: 
The group $G_1$ is not torsion free and the commutation relation 
is not transitive neither in $G_2$ nor in $G_3$.
\end{example}

\subsection{Some HNN-extensions in $\Ext(A,G)$}\label{sec:hnn}
We continue with the assumption that $A = \Z[t]$.  
In \cite{MRS05} a  power $x^t$ with length $\abs{x}\cdot t $
is constructed for  $x \in \cdr(A,\SS)$. (The partial monoid $\cdr(A,\SS)$
has been defined in \refsec{sec:cdr}.) The construction of $x^t$
fails however to satisfy ${\ov x}^t  = \ov{x^t}$, in general. Thus, $x^t$ 
cannot be used to define an HNN extension with stable letter $x^t$.
We content ourselves to prove the following fact. 
\begin{proposition}\label{abel}
Let $x \in W(A,G)$ be  a non-empty cyclically \Gred word.
Then we can define a free abelian subgroup $X$ of $\Ext(A,G)$ with countable  basis $\set{x_d}{d \in \N}$ such that $x_0 = x$.  Hence, the \homo   
$${a_0 + a_1t + \cdots + a_nt^n} \mapsto x^{a_0}(x_1)^{a_1} \cdots (x_n)^{a_n}$$
embeds the abelian group $A$ into $\Ext(A,G)$.
\end{proposition}

\begin{proof}
 Let $\deg(x) = e \geq 0$ and $\abs{x} = \alp$. 

For $k \in \N$ consider $x^{2k}$ as a mapping $x^{2k}:[-k \alp +1 ,k\alp]\to \GG$. We can extend this to a partial (non-closed) word 
$x^{\Z}:D\to \GG$, where the domain is $D= \set{\del \in \Z[t]}{\deg(\del) \leq e}$. Note that $\ov{x}^{\Z}(\del) = x^{\Z} (-\del+1)^{-1}$ for $\del \in D$.

We define $x_{A}: A \to \GG$ as follows: We let
$x_{A}(\bet) = x^{\Z}(\bet)$ for $\deg(\bet) \leq e$. 
For  $\deg(\bet) > e$ write $\bet = t^{e+1} \gam + \del$  with  $\del \in D$; and let $x_{A}(\bet) = x^{\Z}(\del)$. 

Finally, let $x_0 = x$ and  for every $d\geq 1$ let 
$x_d$ be the restriction of $x_{A}$ to the closed interval $[1,t^{e+d}]$.

The word $x_d$ has length $t^{e+d}$ and $\abs{x}$ is a proper period. We have to show that  
 $\ov {x_d}  = (\ov{x})_d$ for $d\geq 1$. To see this, consider $1 \leq \bet \leq t^{e+d}$ and write $\bet = t^{e+1} \gam + \del$ with $\del \in D$. Then: 
 \begin{align*}
\ov {x_d}(\bet)   &=  {x_d}(t^{e+d} -t^{e+1} \gam - \del +1)^{-1} \\
                  &=  {x^{\Z}}( - \del +1)^{-1}\\
                  &=  \ov{x}^{\Z}(\del)\\
                  &=  (\ov{x})_d(\bet)
\end{align*}

Thus, ${a_0 + a_1t + \cdots + a_nt^n} \mapsto x^{a_0}(x_1)^{a_1} \cdots (x_n)^{a_n}$ is a \homo of abelian groups.

 Assume $f(t) = {a_0 + a_1t + \cdots + a_nt^n} \mapsto 1 \in \Ext(A,G)$, then 
 $a_n = 0$ due to the degrees and the fact that $x$ is a cyclically \Gred word.
 By induction $f(t) = 0$. 
\end{proof}

We say that a non-empty word $w \in W(A,G)$ is \emph{primitive} if 
first $w$ does not appear as a  factor of $ww$ other than  as its prefix or as its and second $\ov w$ is not a factor of $ww$.
In particular, a primitive word does not have any  non-trivial proper period.
If on the other hand, we we can write $ww = pwq$  with $1 \leq \abs p < \abs w$, then 
 $\abs p$ is a non-trivial period of $w$.
 Note that the word $w$  which 
looks like 
 $[ababab\cdots\;)(\;\cdots ababab]$ has period 2, it is not primitive, but it is no power
 of any other element. Hence, unlike to the case of finite words, being primitive is 
 a stronger condition than not being a power of any other element.\footnote{A power is an element $u^k$ for $k \in \Z$ since we have not defined $u^\alp$ for $\deg(\alp) > 0$. 
 However even in a more general context the assertions remain true: assume $w$ and $w'$ look like $[ababab\cdots\;)(\;\cdots ababab]$ with
$w = (ab)^\alp$  and  $w' = (ab)^\bet$, where $\abs w = t$ and 
 $\abs {w'} = t+1$. Then we should expect that $(ab)^{\bet - \alp}$ is a power
 of $ab$, but this is not compatible with $\abs{(ab)^{\bet - \alp}}=1.$}

Note also that $ww = p\ov w q$ means that we can write $w = pq$ with 
$\ov p = p$ and $\ov q = q$. It follows that  
$w$ is primitive \IFF $\ov w$ is primitive. 
For a non-abelian free group 
 $F(\SS)$ primitive cyclically reduced words of every positive length exist: Consider $w$ with $w(1)= a$ and $w(\bet)= b$ otherwise.

 Let $H$ be a subgroup of $\Ext(A,G)$ and $u \in H$ be a cyclically \Gred element. As usual the \ei{centralizer} of $u$ in $H$ is the subgroup 
 $\set{v\in H}{uv = vu}$. 
\begin{proposition}\label{HNN}
Let $H$ be a finitely generated subgroup of $\Ext(A,G)$ and let 
$u$, $v$, $w \in H$ be (not necessarily different) cyclically \Gred elements such that $\abs{u} = \abs{v} =\abs{w}$
and such that $w$ is primitive. In addition, let $u$ and $v$ have 
cyclic centralizers in $H$. 
Then the HNN extension
$$H' = \gen{H,t\mid s^{-1} u s = v} $$
embeds into $\Ext(A,G)$. 
\end{proposition}

\begin{proof}Let $\deg(u) = e$. We have $e \geq 0$. 
 Since  $H$ is finitely generated, there is a degree $d$ (with $d > e$) such that
 $deg(x) < d$ for all $x \in H$. By the construction
 according to \refprop{abel} we  define
 the following elements 
 $U= u_{d-e-1}$, $V=v_{d-e-1}$, and $W=w_{d-e-1}  \in \Ext(A,G)$. 
 Recall that $\abs U = \abs V = \abs W = \abs w$ for 
 $d= e+1$ or  $\abs U = \abs V = \abs W = t^{d-1}$ for $d> e+1$. 
 The abelian group of proper periods $\Pi(W)$ 
 is trivial or it 
 is generated by $\abs w$ and  $t^{e+1}, \ldots, t^{d-2}$. 
 The groups $\Pi(U)$ and $\Pi(V)$ may have larger rank than $d-e$. 
 
 Let us define a word $s$ of length $2t^{d}$
  which is depicted as follows:
  $$s  = [UUU\cdots\;)(\;\cdots WW \cdots\;)(\;\cdots VVV].$$
 The group $\Pi(s)$ is 
 generated by $\abs w$ and  $t^{e+1}, \ldots, t^{d-1}$.
  As $u$ is a prefix of $U$, $v$ is a suffix of $V$, and $\abs{u} = \abs{v} =\abs{w}$ is a proper period of $s$, we see that 
  $us = sv$. 
Thus, we obtain a canonical \homo $\phi: H' \to \Ext(A,G)$.
We have to show that $\phi$ is injective. For this it is enough to consider 
a \Breduced word in $H'$ which begins with $s$ or with $\ov s$. 
We can write this word as a sequence $s^{\eps_1}y_1 \cdots s^{\eps_n}y_n$
with $\eps_i = \pm 1$ and $y_i \in H$ for $1 \leq i \leq n$ and we may assume that
$n \geq 1$. 

If the  word is trivial in $\Ext(A,G)$, then it must contain a factor of the 
form $\ov xz x$ where $\deg(z) < \deg(x) = \deg(s)$, $\abs x$ has leading coefficient 1, and $z = 1 \in \Ext(A,G)$. 
Moreover, (by symmetry and by making $x$ shorter if necessary) we may assume that  $x$ or $\ov x$ can be depicted  as $[UUU\cdots\;)(\;\cdots WWW]$.
No such factor $\ov xz x$ appears inside $s$ or $\ov s$. Thus, we have $n \geq 2$ and we 
may assume that $\ov xz x$ is a factor of $s^{\eps_1}y_1 s^{\eps_2}$. 
Assume that $\eps_1 = \eps_2$, say $\eps_1 = \eps_2= 1$, then  $\ov xz x$
appears inside 
$$[UUU\cdots\;)(\;\cdots WW \cdots\;)(\;\cdots VVV]y_1 [UUU\cdots\;)(\;\cdots WW \cdots\;)(\;\cdots VVV].$$

It is also clear that the  factor $z$ must match some factor inside 
the middle part 
$(\;\cdots VVV]y_1 [UUU\cdots\;)$. 
But the word $w$ is primitive, hence $\ov w$ is no factor of $ww$
and  $w$ is a factor of $\ov w \ov w$. Therefore 
this is actually impossible. 

Note that the arguments remain valid even if  e.g.{} $u = \ov{v}$ (which is the 
least evident case). Then $U = \ov{V}$ and infinitely many cancellations inside $s^2$ are possible, but nevertheless 
inside $$[\ov V\ov V \ov V\cdots\;)(\;\cdots WW \cdots\;)(\;\cdots VVV]
[\ov V\ov V \ov V\cdots\;)(\;\cdots WW \cdots\;)(\;\cdots VVV]$$
there is no
factor $\ov{x} z x$ with degree $\deg(s) = \deg (x) > \deg(z)$. 

The conclusion is $\eps_1 =  - \eps_2$
and we may assume $\eps_1 =  - 1$. 
We therefore may assume that $\ov xz x$ is a factor inside the word $\ov sy s $
with $y = y_1 \in H$. Making $z$ longer and $x$ shorter we may assume that $y$ is a factor of the word $z$, and $z$ has the form $\ov{U_1} y U_2$ where $U_1, U_2$ are prefixes of 
$[UUU \cdots\;)$. Without restriction we have $U_1 = U^{n}$ and $\abs{U_1} \leq \abs{ U_2}$. 
Since  $\ov x$ appears as  a suffix of $\ov s$ we may indeed assume that 
$x$ has the form $[UUU \cdots\;)(\;\cdots WWW]$. 
The word  
$x$ begins (inside the word $s$) with $ p uu \cdots$, where $\abs p < \abs u$. 
More precisely, $\abs{U_2}$ is a proper period of $x$, and we can write
$\abs{U_2} = \bet t^{e+1} + m \abs u -\abs p$ for some $\bet \in \Zt$, $m \in \Z$, and suffix $p$ of $u$.
By \reflem{lem:proper} $\abs p$ is a proper period of $x$ and in turn
$\abs p$ is a period of the word $ww$. 
 Since  $w$ is primitive we conclude  $p = 1$, thus $\abs{U_2} = \bet t^{e+1} + m \abs u$. In particular, $U_2$ ends with $(\; \cdots uuu]$ and we see that actually $U_1 = U^{n}$ is a suffix of $U_2$. Replacing $x$ by $U^{n}x$ 
we may assume that the factor  $z$ has the form $y U'$. We conclude that  
$U'$ is a  prefix of  $[UUU \cdots\;)$ and  $U'\in H$
(because $y\in H$ and $z = 1 \in \Ext(A,G)$). 

It is now enough to show that $U' \in \gen{u}$.  
Write $\abs{U'} \equiv \alp \bmod t^{e+1}$  with $\deg(\alp)\leq e$. 
 Note that 
$\abs {U'}= \abs {U_2} - n\abs {U}$ is still a proper period of $x$. Thus, as above we see that  $\alp = k \abs{u}$ for some $k \in \Z$. This implies that $U'$ is in the centralizer of $u$
 and $U'$ is  cyclically \Gred. In particular, $\deg{U'^m} = \rdeg{U'^m}$
 for all $m \in \Z$. By hypothesis the centralizer of $u$ is 
 cyclic, hence for some element 
 $r \in \Ext(A,G)$ and some $\ell, m \in \Z$ we obtain $U' = r^\ell$,  $u= r^m$.
 It follows $U'^m = u^\ell \in \Ext(A,G)$. 
 Hence   $\deg{U'} \leq  e = \deg(u)$, too. 
 We conclude $$\abs{U'}= \alp = k \abs{u}.$$ 
As $U'$ is a prefix of $[UUU \cdots\;)$, we see that 
$U'= u^k$; and  the result is shown. 
\end{proof}
%

%

\begin{thebibliography}{10}

\bibitem{AB87}
R.~{Alperin} and H.~{Bass}.
\newblock Length functions of group actions on {$\Lambda$}-trees.
\newblock {\em Combinatorial group theory and topology}, 111:265--378, 1987.

\bibitem{bass91}
H.~Bass.
\newblock Group actions on non-{A}rchimedean trees.
\newblock In {\em Arboreal group theory ({B}erkeley, {CA}, 1988)}, volume~19 of
  {\em Math. Sci. Res. Inst. Publ.}, pages 69--131. Springer, New York, 1991.

\bibitem{BF}
M.~{Bestvina} and M.~{Feighn}.
\newblock {Stable actions of groups on real trees}.
\newblock {\em Invent. Math.}, 121(2):287--321, 1995.

\bibitem{bo93springer}
R.~Book and F.~Otto.
\newblock {\em Confluent String Rewriting}.
\newblock Springer-Verlag, 1993.

\bibitem{Chis01}
I.~{Chiswell}.
\newblock {\em Introduction to {$\Lambda$}-trees}.
\newblock World Scientific, 2001.

\bibitem{ChMu}
I.~{Chiswell} and T.~{Muller}.
\newblock Embedding theorems for tree-free groups.
\newblock Under consideration.

\bibitem{ddm10}
V.~Diekert, A.~J. Duncan, and A.~G. Myasnikov.
\newblock Geodesic rewriting systems and pregroups.
\newblock In O.~Bogopolski, I.~Bumagin, O.~Kharlampovich, and E.~Ventura,
  editors, {\em Combinatorial and Geometric Group Theory}, Trends in
  Mathematics, pages 55--91. Birkh\"auser, 2010.

\bibitem{BMR95}
A.~M. G.~Baumslag and V.~Remeslennikov.
\newblock Residually hyperbolic groups.
\newblock {\em Proc. Inst. Appl. Math. Russian Acad. Sci.}, 24:3--37, 1995.

\bibitem{GaborioLP94}
D.~{Gaboriau}, G.~{Levitt}, and F.~{Paulin}.
\newblock Pseudogroups of isometries of {$\mathbb R$} and {R}ips' theorem on
  free actions on {$\mathbb R$}-trees.
\newblock {\em Israel. J. Math.}, 87:403--428, 1994.

\bibitem{jan88eatcs}
M.~Jantzen.
\newblock {\em Confluent String Rewriting}, volume~14 of {\em EATCS Monographs
  on Theoretical Computer Science}.
\newblock Springer-Verlag, 1988.

\bibitem{KMI98}
O.~{Kharlampovich} and A.~{Myasnikov}.
\newblock Irreducible affine varieties over a free group. {I}: {I}rreducibility
  of quadratic equations and {N}ullstellensatz.
\newblock {\em J. of Algebra}, 200:472--516, 1998.

\bibitem{KMII98}
O.~{Kharlampovich} and A.~{Myasnikov}.
\newblock Irreducible affine varieties over a free group. {II}: {S}ystems in
  triangular quasi-quadratic form and description of residually free groups.
\newblock {\em J. of Algebra}, 200(2):517--570, 1998.

\bibitem{KMIII05}
O.~{Kharlampovich} and A.~{Myasnikov}.
\newblock Implicit function theorems over free groups.
\newblock {\em J. of Algebra}, 290:1--203, 2005.

\bibitem{KMIV06}
O.~{Kharlampovich} and A.~{Myasnikov}.
\newblock Elementary theory of free non-abelian groups.
\newblock {\em J. of Algebra}, 302:451--552, 2006.

\bibitem{KMRS07}
O.~{Kharlampovich}, A.~{Myasnikov}, V.~{Remeslennikov}, and D.~{Serbin}.
\newblock Groups with free regular length functions in {$\mathbb{Z}^n$}.
\newblock To appear, arXiv:0907.2356v2.

\bibitem{KMRS04}
O.~{Kharlampovich}, A.~{Myasnikov}, V.~{Remeslennikov}, and D.~{Serbin}.
\newblock Subgroups of fully residually free groups: algorithmic problems.
\newblock In A.~G. Myasnikov and V.~Shpilrain, editors, {\em Group theory,
  Statistics and Cryptography}, volume 360, pages 63--101, 2004.

\bibitem{Lynpara60}
R.~Lyndon.
\newblock Groups with parametric exponents.
\newblock {\em Trans. Amer. Math. Soc.}, 9:518–--533, 1960.

\bibitem{LS}
R.~{Lyndon} and P.~{Schupp}.
\newblock {\em Combinatorial Group Theory}.
\newblock Classics in Mathematics. Springer, 2001.

\bibitem{Lys90}
I.~{Lysenok}.
\newblock On some algorithmic problems of hyperbolic groups.
\newblock {\em Math. USSR Izvestiya}, 35:145--163, 1990.

\bibitem{Mihailova58}
K.~A. {Mihailova}.
\newblock {The occurrence problem for direct products of groups}.
\newblock {\em Dokl. Akad. Nauk SSSR}, 119:1103--1105, 1958.
\newblock English translation in: {M}ath. {USSR} {S}bornik, 70: 241--251, 1966.

\bibitem{MoSh84}
J.~{Morgan} and P.~{Shalen}.
\newblock Valuations, trees, and degenerations of hyperbolic structures.
\newblock {\em I. Annals of Math}, 120(3):401--476, 1984.

\bibitem{MRS05}
A.~{Myasnikov}, V.~{Remeslennikov}, and D.~{Serbin}.
\newblock Regular free length functions on {L}yndon's free
  {$\mathbb{Z}[t]$}-group {$F^{\mathbb{Z}[t]}$}.
\newblock {\em Contemp. Math., Amer. Math. Soc.}, 378:37--77, 2005.

\bibitem{nikolaev10thesis}
A.~Nikolaev.
\newblock {\em Membership Problem in Groups Acting Freely on Non-Archimedean
  Trees}.
\newblock Doctor of philosophy, McGill University, Montreal, Quebec, August
  2010.

\bibitem{olsap00}
A.~Y. Olshanskii and M.~V. Sapir.
\newblock Length functions on subgroups in finitely presented groups.
\newblock In {\em Groups --- Korea'98 (Pusan)}. de Gruyter, 2000.

\bibitem{olsap01}
A.~Y. Olshanskii and M.~V. Sapir.
\newblock Length and area functions on groups and quasi-isometric {H}igman
  embeddings.
\newblock {\em IJAC}, 11(2):137--170, 2001.

\bibitem{rips82}
E.~Rips.
\newblock Subgroups of small cancellation groups.
\newblock {\em Bull. London Math. Soc.}, 14:45--47, 1982.

\bibitem{romanovskii}
N.~S. {Romanovskii}.
\newblock {The occurrence problem for extensions of abelian by nilpotent
  groups}.
\newblock {\em Sib. Math. J.}, 21:170--174, 1980.

\bibitem{serre80}
J.-P. {Serre}.
\newblock {\em Trees}.
\newblock New York, Springer, 1980.

\bibitem{Stallings71}
J.~Stallings.
\newblock {\em Group theory and three-dimensional manifolds}.
\newblock Yale University Press, New Haven, Conn., 1971.
\newblock A James K. Whittemore Lecture in Mathematics given at Yale
  University, 1969, Yale Mathematical Monographs, 4.

\end{thebibliography}
\newcommand{\Ju}{Ju}\newcommand{\Ph}{Ph}\newcommand{\Th}{Th}\newcommand{\Ch}{C%
h}\newcommand{\Yu}{Yu}\newcommand{\Zh}{Zh}

{} \\
\small 
\noindent
Volker Diekert, 
Universit\"at Stuttgart,
Universit\"atsstr. 38, 70569 Stuttgart, Germany\\ 
\noindent
Alexei Myasnikov,
Stevens Institute of Technology,
Hoboken, NJ 07030, USA

\end{document}